\documentclass[12pt,a4paper]{amsart}
\usepackage{geometry}
\usepackage{amsmath}
\usepackage{amssymb}
\usepackage{amsfonts}
\usepackage{mathrsfs}
\usepackage{graphicx}
\usepackage{xcolor}
\usepackage{enumitem}
\usepackage{bbm}
\usepackage[hidelinks,colorlinks=black,linkcolor=black,citecolor=black]{hyperref}
\usepackage[noabbrev,capitalize]{cleveref}

\newtheorem{theorem}{Theorem}[section]

\newtheorem{corollary}[theorem]{Corollary}

\newtheorem{definition}[theorem]{Definition}
\newtheorem{example}[theorem]{Example}

\newtheorem{lemma}[theorem]{Lemma}

\newtheorem{question}[theorem]{Question}
\newtheorem{proposition}[theorem]{Proposition}
\newtheorem{remark}[theorem]{Remark}

\begin{document}

\title{Some questions related to free-by-cyclic groups and tubular groups}

\author{Xiaolei Wu}
\address{Shanghai Center for Mathematical Sciences, Jiangwan Campus, Fudan University, No.2005 Songhu Road, Shanghai, 200438, P.R. China}
\email{xiaoleiwu@fudan.edu.cn}

\author{Shengkui Ye}
\address{NYU Shanghai, No.567 Yangsi West 
  Rd, Pudong New Area, Shanghai, 200124, P.R. China \\
NYU-ECNU Institute of Mathematical Sciences at NYU Shanghai, 3663 Zhongshan Road North, Shanghai, 200062, China}
\email{sy55@nyu.edu}


\date{December, 2025}


\maketitle

\begin{abstract}
We prove that a CAT(0) free-by-cyclic tubular group with a single vertex is
virtually special, but many of them cannot virtually act freely and cocompactly on CAT(0) cube complexes. This partially confirms a question of
Brady--Soroko \cite[Section 9: Question 1]{BS} and answers a question of Lyman \cite[Question 1]{Ly} in the negative. Furthermore, we provide examples of free-by-cyclic groups
amalgamated along cyclic subgroups that are not virtually free-by-cyclic. This answers
negatively a question of Hagen--Wise \cite[Remark 3.6]{hw}. Lastly, we exhibit an example of a cyclic-subgroup-separable tubular group that does not have the property (VRC) (i.e. every cyclic subgroup is a virtual retract). This answers a question of Minasyan \cite[Question 11.6]{min} in the negative. 
\end{abstract}


\section{Introduction}

This paper aims to study some questions that lie at the intersection of free-by-cyclic groups and tubular groups. Recall that a group $G$ is called free-by-cyclic if it surjects to $\mathbb{Z}$ with a free kernel (of possibly infinite rank). If the kernel is a free group of rank $n,$ we call $G$ an $F_{n}$-by-$\mathbb{Z}$ group. A tubular group is a group that splits as a finite graph of groups with $\mathbb{Z}^2$ vertex groups and $\mathbb{Z}$ edge groups. There are many examples of free-by-cyclic groups that are tubular. For example, Gersten's famous non-$\mathrm{CAT(0)}$ example is tubular \cite{ge}.    Button gives a criterion to determine whether a tubular group is free-by-cyclic \cite[Theorem 2.1]{bu}.  The first part of our study is motivated by the following two questions:

\begin{question} \cite[Section 9, Question 1]{BS}
\label{prob1}Does every $\mathrm{CAT(0)}$ free-by-cyclic group virtually embed in a right-angled Artin group $(\mathrm{RAAG})$?
\end{question}
\begin{question} \cite[Problem 5.3]{page}
\label{prob1.1} What’s the relationship between being $\mathrm{CAT(0)}$ and (virtually) $\mathrm{RFRS}$ for free-by-cyclic groups?
\end{question}

Since a virtually special group (i.e., a subgroup of a RAAG) is virtually RFRS \cite[Corollary 2.3]{Ag08}, a positive answer to Question \ref{prob1} would solve Question \ref{prob1.1} in one direction. Recall that it is already a hard question when a free-by-cyclic group is $\mathrm{CAT(0)}$ (see for example 
Bridson's problem list \cite[Question 7.9]{Bri}). The far-reaching work of Agol \cite{Ag} and Wise \cite{wi12},
together with Hagen--Wise \cite{ha1,ha2} implies that a hyperbolic
$F_n$-by-$\mathbb{Z}$ group is virtually compact special. Button
and Kropholler \cite{bk} prove that every $F_{2}$-by-$\mathbb{Z}$ group is
the fundamental group of a non-positively curved cube complex of dimension
2. On the other hand, every $F_{2}$-by-$\mathbb{Z}$ group can be represented as a
non-positively curved punctured-torus bundle over the circle, implying that every $%
F_{2} $-by-$\mathbb{Z}$ group is $\mathrm{CAT(0)}$ and virtually special
by the work of Liu \cite[Theorem 1.1]{liu}. When $n\ge 3$, Samuelson \cite{sa} proves that the $F_{n}\rtimes
_{\phi }\mathbb{Z}$ groups are $\mathrm{CAT(0)}$ for some upper triangular
automorphism $\phi$.  Lyman \cite{Ly} further provides several infinite families
of polynomially growing automorphisms of free groups whose mapping tori are
$\mathrm{CAT(0)}$. In a recent paper, Munro and Petyt \cite{mp} discuss the compact specialness
of free-by-cyclic groups using the median metric space. On the other hand, the study of cubulating tubular groups is in a much better shape.  In fact, Wise in \cite[Theorem 1.1]{wi14} gives an easy criterion to determine when a tubular group acts freely on a $\mathrm{CAT(0)}$ cube complex. Woodhouse provides a way to produce free actions of tubular groups on finite dimensional $\mathrm{CAT(0)}$ cube complexes \cite[Theorem 1.2]{Wood2} and proves further in \cite[Theorem 1.1]{Wood} that a tubular group is virtually special if and only if it acts freely on a finite dimensional $\mathrm{CAT(0)}$ cube complex.

Now let $F_3=\langle a,b,c \rangle$ be a free group of rank 3. Consider the following automorphism $\Psi: F_3 \rightarrow F_3:$
$$
a \mapsto a, b \mapsto a^m b a^m, c  \mapsto  a^n c a^n $$
for  integers $m,n$.  Lyman \cite{Ly} proves that the corresponding group $F_{3}\rtimes _{\Psi }\mathbb{Z}$ (among others) is $\mathrm{CAT(0)}$.
Based on Question \ref{prob1}, Lyman \cite[Question 1.1]{Ly} further asks the
following question:

\begin{question}
\label{prob2} May these $\mathrm{CAT(0)}$ free-by-cyclic groups (eg. $F_{3}\rtimes _{\Psi }\mathbb{Z}$) be cocompactly
cubulated? Is the resulting cube complex virtually (co-)special?
\end{question}
We  answer Question \ref{prob2} negatively by viewing it as a tubular group. 
\begin{theorem}\label{thm:Lyman} [Example \ref{eg1}]
  For any integers $m,n$, the group $F_{3}\rtimes _{\Psi }\mathbb{Z}$ is virtually special. But  it virtually acts on a CAT(0) cube complex freely and cocompactly if and only if $|m| = |n|$. 
\end{theorem}

Button  (cf. \cite{but} remarks after Corollary 3) also notices that there exist non-compact virtually special tubular groups whose underlying graph is a
tree. More generally, we determine the virtual (compact) specialness for free-by-cyclic groups of the form: 

\[G(\{p_{i},q_{i}\}_{i=1}^{n-1}) =\langle  a_{0},a_{1},...,a_{n-1}, t\mid ta_0t^{-1} =a_0, ta_it^{-1} = a_0^{p_{i}}a_{i}a_0^{q_{i}}, 1\leq i\leq n-1 \rangle,\]
where $p_{1},\cdots ,p_{n-1},q_{1},\cdots ,q_{n-1}$ are integers; see Theorem \ref{th7-1} and Theorem \ref{cor:vir-cpspe}.

The second part of our study considers the amalgamation of two free-by-cyclic groups. According
to Hagen--Wise \cite[Remark 3.6]{hw}, it is not clear whether the class of
free-by-cyclic groups is closed under amalgams along $\mathbb{Z}$ subgroups.
 Example \ref{eg2} will show that the answer to their question is negative, we therefore revise their question to the following.

\begin{question}
\label{prob3} Is the amalgamated product  of two free-by-cyclic groups over a
cyclic subgroup virtually free-by-cyclic?
\end{question}

It turns out that the answer to Question \ref{prob3} is still
negative.

\begin{theorem}[\Cref{corlast2}]
  Let $G_1 =\langle a,b,c,t\mid tat^{-1}=a,tbt^{-1}=ba, tct^{-1} =ca^2\rangle $ be the Gersten group, $G_2 = \langle x,s\mid sxs^{-1} = x \rangle \cong \mathbb{Z}^2$. Then the tubular group $G = G_1 \ast_{a=x} G_2$ is not virtually free-by-cyclic. 
\end{theorem}
\begin{remark}
Kielak and Linton in \cite[Theorem 1.1]{KL24} give a criterion for a hyperbolic and virtually compact special group to be virtually free-by-cyclic:  it is virtually free-by-cyclic if and only if $b_2^{(2)}(G)=0$ and $\mathrm{cd}_{\mathbb{Q}}\leq 2$. Fisher improves their criterion by relaxing the condition ``hyperbolic and virtually compact special" to just ``$\mathrm{RFRS}$" \cite[Theorem A]{Fi24}. On the other hand, it is not too hard to see that tubular groups are $2$-dimensional, locally indicable, coherent \cite[Theorem 5.1]{Wi20}, and have vanishing $\ell^2$-Betti numbers. Also, our group $G$ does not have any Baumslag--Solitar subgroup of the form $\langle a,t\mid ta^mt^{-1} =a^n\rangle$ where $|m|\neq |n|$. This shows that the $\mathrm{RFRS}$ condition in Fisher's theorem cannot be easily further improved. 
\end{remark}

The last part of our study considers the property $\mathrm{(VRC)}$.  Recall that Minasyan \cite{min} following Long--Reid \cite{lr} defines a group $G$ to
be $\mathrm{(VRC)}$ (virtual retract onto any cyclic subgroup) if for every cyclic subgroup $C$,
there is a finite index subgroup $H\leq G$ such that $H\geq C$ and $H$ has
an epimorphism $H\rightarrow C$ satisfying $C\rightarrow
H\rightarrow C$ is the identity. A subgroup $K$ of a group $G$ is called separable if $K$ is the intersection of a family of finite index subgroups in $G$. The following question was asked by Minasyan.

\begin{question}\cite[Question 11.6]{min}
\label{Ques-Min}
 Let $G$ be the fundamental group of a finite graph of groups with free abelian
vertex groups (of finite ranks) and cyclic edge groups. If $G$ is cyclic subgroup separable, does
it necessarily satisfy $\mathrm{(VRC)}$?   
\end{question}

The answer to Question \ref{Ques-Min} is again negative. 

\begin{theorem}[Example \ref{eg7.2}]
    The Gersten group $$\langle a,b,c,t \mid tat^{-1}=a,tbt^{-1}=ba, tct^{-1}=ca^2 \rangle $$ is a cyclic-subgroup-separable tubular group that does not have the property $\mathrm{(VRC)}$.
\end{theorem}
    

 In the final section, we highlight the subtle distinction between RFRS and virtually RFRS. We show that the non-virtual analogue of Question~\ref{prob1.1} has a negative answer by exhibiting a CAT(0) free-by-cyclic group that is not RFRS. We also observe that Agol's question \cite[Question~11]{Ag14} of whether braid groups are RFRS has a negative answer for braid groups on at least five strands.


\subsection*{Structure of the Paper.} In \S2, we study the CAT(0)ness of  tubular groups with a single vertex group by giving some concrete criterion. In \S3, we give a sufficient and necessary condition for a tubular group to be free-by-cyclic or $F_n$-by-$\mathbb{Z}$ using Button's criterion. We then study the virtual specialness of tubular groups in \S4. We prove that a CAT(0) free-by-cyclic tubular group over a single vertex is virtually special. In \S5, we discuss when amalgamated products of two free-by-cyclic groups along cyclic subgroups are again (virtually) free-by-cyclic.  In \S6, we study the (VRC) property and subgroup separation of free-by-cyclic groups. In the final section, we discuss the subtle difference between  RFRS and virtually RFRS, then propose some questions.

\subsection*{Acknowledgments.}
Wu is currently a member of LMNS and is supported by NSFC No.12326601. Ye is
supported by NSFC No. 11971389. We thank Ian Agol, Sami Douba, Mark Hagen, Monika Kudlinska and  Daniel Wise for their helpful communications. In particular, we learned the idea of proof for Lemma \ref{lem6.11} from Sami Douba. After we shared our preprint with Ashot Minasyan, he informed us that in his on-going joint work with Jon Merladet, they also found many counterexamples to  \Cref{Ques-Min}. We also thank him for many comments on our preprint.

\section{$\mathrm{CAT(0)}$ness}\label{section:CAT(0)}

Recall that a group $G$ is $\mathrm{CAT(0)}$ if it can act properly and cocompactly on a $\mathrm{CAT(0)}$ space by isometries. The CAT(0)ness of a tubular group has been studied by Button \cite[Theorem 4]{but1}, where he gives a criterion in terms of translation length assignments. We study  tubular groups with only one vertices by providing some concrete conditions so that they are $\mathrm{CAT(0)}$, free-by-cyclic, or virtually special.  The proof of the CAT(0)ness part is standard which one may also  deduce from Button's criterion.  

To motivate our study, let us first recall the  groups $G=G(\{p_{i},q_{i}\}_{i=1}^{n-1})=F_{n}\rtimes _{\phi }%
\mathbb{Z}$ from the introduction:  $F_{n}$ is the free group generated by $%
\{a_{0},a_{1},...,a_{n-1}\}$ and $\phi :F_{n}\rightarrow F_{n}$ be a
homomorphism given by 
\begin{eqnarray*}
a_{0} &\mapsto &a_{0}, \\
a_{i} &\mapsto &a_0^{p_{i}}a_{i}a_0^{q_{i}},i=1,2,...,n-1,
\end{eqnarray*}%
for some integers $p_{1},\cdots ,p_{n-1},q_{1},\cdots ,q_{n-1}$. Note that
the semi-direct product $G=F_{n}\rtimes _{\phi }\mathbb{Z}$ has a
presentation:
\begin{equation*}
\langle
a_{0},a_{1},...,a_{n-1},t:ta_{0}t^{-1}=a_{0},ta_{i}t^{-1}=a_{0}^{p_{i}}a_{i}a_{0}^{q_{i}},i=1,2,...,n-1\rangle .
\end{equation*}%
Rewrite the relator $ta_{i}t^{-1}=a_{0}^{p_{i}}a_{i}a_{0}^{q_{i}}$ as 
$a_{i}a_{0}^{q_{i}}ta_{i}^{-1} = a_{0}^{-p_{i}}t$ for each $i.$ One sees that $G$ can be viewed as a
tubular group with the vertex group $G_{v}=\mathbb{Z}^{2}=\langle a_{0},t\rangle 
$ and stable letters $a_{i}$ for $1\leq i\leq n-1$. The group $G$ is
not hyperbolic since it contains $\mathbb{Z}^{2}$ as a subgroup. When $%
n=3,p_{1}=p_{2}=0$, the group $G$ is studied by Gersten \cite{ge}, who
proves that when $q_{1}\neq q_{2}$ are positive integers, the group $G$ is not $\mathrm{CAT(0)}$ (actually, $G$
cannot act properly by semi-simple isometries on any $\mathrm{CAT(0)}$ metric
space). When $p_{i}= \pm  q_{i}$,
the group $G$ is studied by Lyman, who proves that they
are $\mathrm{CAT(0)}$; see \cite{Ly} for more details and more examples of CAT(0) free-by-cyclic groups. 

Let us now focus on tubular groups with one vertex group.

\begin{theorem}
\label{main}Let $G=\langle \mathbb{Z}%
^{2},s_{i}|s_{i}v_{i}s_{i}^{-1}=w_{i},i=1,2,...,k\rangle$, where $v_i,w_i$ are nontrivial vectors in $\mathbb{Z}^2$.

\begin{enumerate}[label=(\arabic*)]
\item Suppose that $\{v_{1},w_{1}\}$ are linearly independent columns of the invertible rational matrix $[v_1,w_1]$, and $%
[v_{1},w_{1}]^{-1}v_{i}=%
\begin{pmatrix}
x_{i} \\ 
y_{i}%
\end{pmatrix}%
,[v_{1},w_{1}]^{-1}w_{i}=%
\begin{pmatrix}
x_{i}^{\prime } \\ 
y_{i}^{\prime }%
\end{pmatrix}%
,i=2,...,k$. The group $G$ is $\mathrm{CAT(0)}$ if and only if 
\begin{equation*}
x_i^2+y_i^2+2x_iy_i\cos\theta
=
(x_i')^2+(y_i')^2+2x_i'y_i'\cos\theta,
\qquad i=2,\ldots,k,
\end{equation*}%
for some common $\theta \in (0,\pi ).$

\item The group $G$ is free-by-cyclic if and only if $v_{1}-w_{1}$, $%
v_{2}-w_{2},...,v_{k}-w_{k}$ lie in a common line $l$ which doesn't contain
any element of $\{v_{1},v_{2},...,v_{k}\}.$

\item The group $G$ is virtually special if $\{v_{1},w_{1}\}$ are
linearly independent and 
\begin{eqnarray*}
\det [w_{1}-v_{1},v_{i}] &=&\pm \det [w_{1}-v_{1},w_{i}], \\
\det [w_{1}+v_{1},v_{i}] &=&\pm \det [w_{1}+v_{1},w_{i}],i=2,...,k.
\end{eqnarray*}
\end{enumerate}
\end{theorem}

\begin{remark}
    A tubular free-by-cyclic group is actually $F_{n}$-by-$\mathbb{Z}$ (see Proposition \ref{th10}).
\end{remark} 
The proof of Theorem \ref{main} will be spread in the following sections.

Let $v_{1},v_{2},...,v_{k},w_{1},w_{2},...,w_{k}\in \mathbb{Z}^{n}$ be $2k$
nonzero integer vectors and $G=\langle \mathbb{Z}%
^{n},s_{i}|s_{i}v_{i}s_{i}^{-1}=w_{i},i=1,2,...,k\rangle $ be a multiple
HNN extension.  The following is a starting point.

\begin{lemma}
\label{lem1}The group $G$ is $\mathrm{CAT(0)}$ if and only if there is an invertible
matrix $A_{n\times n}\in \mathrm{GL}_{n}(\mathbb{R})$ such that 
\begin{equation*}
\parallel Av_{i}\parallel =\parallel Aw_{i}\parallel ,i=1,2,...,k.
\end{equation*}
\end{lemma}



\begin{proof}
Suppose first that $G$ acts properly and cocompactly on a $\mathrm{CAT}(0)$
space $X$. By the flat torus theorem, the subgroup $\mathbb Z^n\leq G$
stabilizes a Euclidean flat $\mathbb E^n\subset X$ such that
$\mathbb E^n/\mathbb Z^n$ is a flat torus. Choose an origin and an
orthonormal basis for the translation space of $\mathbb E^n$, thereby
identifying $\mathbb E^n$ isometrically with $\mathbb R^n$.

Let $\{e_1,\ldots,e_n\}$ be the standard basis of $\mathbb Z^n$, and
suppose that $e_j$ acts on $\mathbb E^n\cong\mathbb R^n$ by translation
through the vector $A_j\in\mathbb R^n$. Then
\[
A=[A_1,\ldots,A_n]\in\mathrm{GL}_n(\mathbb R),
\]
because the translation vectors $A_1,\ldots,A_n$ generate a lattice in
$\mathbb R^n$. An element $v\in\mathbb Z^n$ acts by translation through
the vector $Av$. Since $s_i v_i s_i^{-1}=w_i$, the elements $v_i$ and
$w_i$ have the same translation length in $X$. Hence
\[
\|Av_i\|=\|Aw_i\|,\qquad i=1,\ldots,k.
\]

Conversely, suppose that there is an invertible matrix
\[
A=[A_1,\ldots,A_n]\in\mathrm{GL}_n(\mathbb R)
\]
such that
\[
\|Av_i\|=\|Aw_i\|,\qquad i=1,\ldots,k.
\]
Define an action of $\mathbb Z^n$ on $\mathbb E^n=\mathbb R^n$ by letting
$e_j$ act as
\[
x\longmapsto x+A_j,\qquad j=1,\ldots,k.
\]
Then the translation lengths of $v_i$ and $w_i$ are $\|Av_i\|$ and
$\|Aw_i\|$, respectively, and hence are equal. Therefore, the multiple
HNN extension is a $\mathrm{CAT}(0)$ group by
\cite[Theorem~II.11.18]{BH}.
\end{proof}

\begin{remark}
\label{remark}The necessary condition of the  theorem holds for any
$\mathrm{CAT(0)}$ group containing $G.$ In particular, a tubular group is $\mathrm{CAT(0)}$ only
if for each vertex $v$ there is an invertible matrix $A_{2\times 2}\in 
\mathrm{GL}_{2}(\mathbb{R})$ such that 
\begin{equation*}
\parallel Av_{i}\parallel =\parallel Aw_{i}\parallel ,i=1,2,...,k.
\end{equation*}%
Here $v_i$ and $w_i$ come from the edge relation $
s_{i}v_{i}s_{i}^{-1}=w_{i}$ for each edge $s_{i}$ connecting $v$ to $v$.
\end{remark}

When $k=1,$ we have the following.

\begin{corollary}
The group $G=\langle \mathbb{Z}^{n},s|svs^{-1}=w\rangle $ is $\mathrm{CAT(0)}$ if and
only if $\{v,w\}$ is linearly independent or $v=\pm w.$
\end{corollary}

\begin{proof}
If $\{v,w\}$ is linearly independent, there is an invertible real matrix $%
A_{n\times n}$ containing $v,w$ as its first two columns. Then $%
A^{-1}v,A^{-1}w$ have the same length. If $v=\pm w,$ they already have the
same length. Lemma \ref{lem1} implies that $G$ is $\mathrm{CAT(0)}$.

If $G$ is $\mathrm{CAT(0)}$ and $\{v,w\}$ is linearly dependent, we have $av=bw$ for
some integers $a,b.$ Since $v,w$ are conjugate, they have the same
translation length. Therefore, we have $|av|=|bw|$ implying $v=\pm w.$
\end{proof}

Theorem \ref{main} (1) is the following result.

\begin{theorem}\label{th1}
Let
\[
G=\langle \mathbb{Z}^{2},s_i
\mid s_iv_is_i^{-1}=w_i,\ i=1,2,\ldots,k\rangle.
\]
Suppose that $\{v_1,w_1\}$ is linearly independent and
\[
[v_1,w_1]^{-1}v_i=
\begin{pmatrix}
x_i\\
y_i
\end{pmatrix},
\qquad
[v_1,w_1]^{-1}w_i=
\begin{pmatrix}
x_i'\\
y_i'
\end{pmatrix},
\qquad i=2,\ldots,k.
\]
Then $G$ is $\mathrm{CAT}(0)$ if and only if
\[
x_i^2+y_i^2+2x_iy_i\cos\theta
=
(x_i')^2+(y_i')^2+2x_i'y_i'\cos\theta,
\qquad i=2,\ldots,k,
\]
for some common $\theta\in(0,\pi)$.
\end{theorem}

\begin{proof}
Let $A\in\mathrm{GL}_2(\mathbb R)$ and set
\[
B=A[v_1,w_1]=[B_1,B_2].
\]
Then
\begin{align*}
Av_1
&=B[v_1,w_1]^{-1}v_1
 =Be_1=B_1,\\
Aw_1
&=B[v_1,w_1]^{-1}w_1
 =Be_2=B_2,\\
Av_i
&=B[v_1,w_1]^{-1}v_i
 =B
 \begin{pmatrix}
 x_i\\
 y_i
 \end{pmatrix}
 =x_iB_1+y_iB_2,\\
Aw_i
&=B[v_1,w_1]^{-1}w_i
 =B
 \begin{pmatrix}
 x_i'\\
 y_i'
 \end{pmatrix}
 =x_i'B_1+y_i'B_2
\end{align*}
for $i=2,\ldots,k$.

Let $\theta$ be the ordinary, unoriented Euclidean angle between
$B_1$ and $B_2$. Since $B$ is invertible, the vectors $B_1$ and
$B_2$ are linearly independent, and hence $\theta\in(0,\pi)$. We have the dot product
\[
\langle B_1,B_2\rangle
=
\|B_1\|\|B_2\|\cos\theta.
\]
Moreover,
\[
\|Av_1\|=\|Aw_1\|
\quad\Longleftrightarrow\quad
\|B_1\|=\|B_2\|.
\]
For $i=2,\ldots,k$, we have
\begin{align*}
\|Av_i\|^2
& =\langle A v_i,A v_i\rangle =
x_i^2\|B_1\|^2+y_i^2\|B_2\|^2
+2x_iy_i\|B_1\|\|B_2\|\cos\theta,\\
\|Aw_i\|^2
&=
(x_i')^2\|B_1\|^2+(y_i')^2\|B_2\|^2
+2x_i'y_i'\|B_1\|\|B_2\|\cos\theta.
\end{align*}
Thus, once $\|B_1\|=\|B_2\|$, the equality
$\|Av_i\|=\|Aw_i\|$ is equivalent to
\[
x_i^2+y_i^2+2x_iy_i\cos\theta
=
(x_i')^2+(y_i')^2+2x_i'y_i'\cos\theta.
\]

If $G$ is $\mathrm{CAT}(0)$, Lemma~\ref{lem1} provides an
$A\in\mathrm{GL}_2(\mathbb R)$ such that
\[
\|Av_i\|=\|Aw_i\|,
\qquad i=1,\ldots,k.
\]
The preceding computation therefore gives a common
$\theta\in(0,\pi)$ satisfying the required equations.

Conversely, suppose that such a $\theta\in(0,\pi)$ exists. Choose two
vectors $B_1,B_2\in\mathbb R^2$ of the same nonzero length whose
ordinary angle is $\theta$, and set
\[
A=[B_1,B_2][v_1,w_1]^{-1}.
\]
Since $\theta\in(0,\pi)$, the vectors $B_1$ and $B_2$ are linearly
independent, so $A\in\mathrm{GL}_2(\mathbb R)$. The preceding
computations show that
\[
\|Av_i\|=\|Aw_i\|,
\qquad i=1,\ldots,k.
\]
Therefore, $G$ is $\mathrm{CAT}(0)$ by Lemma~\ref{lem1}.
\end{proof}

\begin{example}\label{exam:Gpq-CAT(0)}
    Consider the tubular group $G = G(\{p_1,q_1\},\{p_2,q_2\})$. It has the following tubular presentation:
    $$\langle a,b,c,t\mid [a,t]=1, b^{-1} a^{-p_1}tb = a^{q_1}t,c^{-1} a^{-p_2}tc = a^{q_2}t \rangle.$$
   Let us first replace $t$ by $t' = a^{-p_1} t$, and get a slightly different tubular presentation for $G$:
     $$\langle a,b,c,t'\mid [a,t']=1, b^{-1} t' b = a^{p_1+ q_1}t',c^{-1} a^{p_1-p_2}t'c = a^{p_1+q_2}t' \rangle.$$
 We will show the $\mathrm{CAT(0)}$ness of $G$ under the assumption that $p_1 - q_1 = p_2 -q_2$. Let $k= p_1 - q_1 = p_2 -q_2$, then $q_1 = p_1-k$ and $q_2 = p_2-k$. Let further $l=p_1-p_2$, then $p_2 = p_1 -l$.  The presentation of $G$ is reduced to  $$\langle a,b,c,t'\mid [a,t']=1, b^{-1} t' b = a^{2p_1-k}t',c^{-1} a^{l}t'c = a^{2p_1-k-l}t' \rangle,$$
    where $p_1,k,l$ can be any integer.  Let $m=2p_1-k =p_1+q_1$, then the presentation of $G$ is further reduced to
    $$\langle a,b,c,t'\mid [a,t']=1, b^{-1} t' b = a^{m}t',c^{-1} a^{l}t'c = a^{m-l}t' \rangle,$$
    for some integers $m,l$. Therefore, $G$ is the tubular group with a single vertex group $\mathbb{Z}^2$ and edge relations given by
 \begin{eqnarray*}
b^{-1}%
\begin{bmatrix}
0 \\ 
1%
\end{bmatrix}%
b &=&%
\begin{bmatrix}
m \\ 
1%
\end{bmatrix}%
, \\
c^{-1}%
\begin{bmatrix}
l \\ 
1%
\end{bmatrix}%
c &=&%
\begin{bmatrix}
m-l\\ 
1%
\end{bmatrix}%
.
\end{eqnarray*}
If $m=0$, then the two pairs of edge vectors are
$(0,1),(0,1)$ and $(l,1),(-l,1)$, respectively. These pairs have
equal Euclidean lengths, so Lemma \ref{lem1} implies that $G$ is
$\mathrm{CAT}(0)$. Thus, we may assume that $m\neq 0$. Then
$$ \begin{pmatrix}
0 & m \\ 
1 & 1
\end{pmatrix}^{-1} \begin{pmatrix}
l \\ 
1 
\end{pmatrix} = \frac{1}{m}\begin{pmatrix}
m-l \\ 
l
\end{pmatrix}, $$
$$ \begin{pmatrix}
0 & m \\ 
1 & 1
\end{pmatrix}^{-1} \begin{pmatrix}
m-l \\ 
1 
\end{pmatrix} = \frac{1}{m}\begin{pmatrix}
l\\ 
m-l
\end{pmatrix}.$$

To show that $G$ is $\mathrm{CAT(0)}$, by Theorem \ref{main} (1), it suffices to find $\theta \in (0,\pi)$ such that 

$$(m-l)^2 + l^2 +2l(m-l)\cos\theta= l^2 + (m-l)^2 +2l(m-l)\cos\theta.$$
Note now that the equality holds for any $\theta \in(0,\pi)$, implying that $G$ is indeed $\mathrm{CAT(0)}$.
\end{example}

\section{Free-by-cyclic groups}

In this section, we study when a tubular group is free-by-cyclic and $F_n$-by-$\mathbb{Z}$. Let us start by proving Theorem \ref{main} (2).
\begin{theorem}
\label{th2}Let $G=\langle \mathbb{Z}%
^{2},s_{i}|s_{i}v_{i}s_{i}^{-1}=w_{i},i=1,2,...,k\rangle$. Then $G$ is
free-by-cyclic if and only if $v_{1}-w_{1}$, $v_{2}-w_{2},...,v_{k}-w_{k}$
lie in a common line $l$ which doesn't contain any element of $%
\{v_{1},v_{2},...,v_{k}\}.$
\end{theorem}

\begin{proof}
It is already known by Button \cite[Theorem 2.1]{bu} that a tubular group is
free-by-cyclic if and only if there exists a homomorphism from $G$ to $%
\mathbb{Z}$, which is non-zero on every edge group.

If $v_{1}-w_{1}$ is parallel to any $v_{i}-w_{i},i=2,...,k,$ and one of them
is nonzero, suppose that $v=%
\begin{pmatrix}
a \\ 
b%
\end{pmatrix}%
$ is a primitive element in $\mathbb{Z}^{2}$ such that $%
v_{1}-w_{1},v_{2}-w_{2},...,v_{k}-w_{k}\in \langle v\rangle .$ Take $%
v^{\prime }=%
\begin{pmatrix}
c \\ 
d%
\end{pmatrix}%
$ be another primitive element such that $\{v,v^{\prime }\}$ is a basis for $%
\mathbb{Z}^{2}$ (for example, choose integers $c,d$ such that $ad-bc=1$). Let 
\begin{equation*}
\mathrm{proj}:\mathbb{Z}^{2}\hookrightarrow \mathbb{R}^{2}\rightarrow \mathbb{Z}
\end{equation*}%
be the projection onto the component $\mathbb{R}v^{\prime }$ (i.e. for any $x=x_1 v+ x_2 v^{\prime}$, we have $\mathrm{proj}(x)=x_2  $). Define 
\begin{eqnarray*}
f &:&G\rightarrow \mathbb{Z}, \\
s_{i} &\rightarrow &0, \\
x &\rightarrow &\mathrm{proj}(x),x\in \mathbb{Z}^{2}.
\end{eqnarray*}%
Since $\mathrm{proj}(v_{i})=\mathrm{proj}(w_{i}),i=1,...,k,$ we know that $f$
is a well-defined group homomorphism. Since the line $\mathbb{R}v$
containing $v_{1}-w_{1},v_{2}-w_{2},...,v_{k}-w_{k}$ does not contain any of 
$\{v_{1},v_{2},...,v_{k}\},$ we know that $\mathrm{proj}(v_{i}),\mathrm{proj}%
(w_{i}),i=1,...,k,$ are nonzero. The criterion of Button proves that $G$ is
free-by-cyclic.

If all $v_{i}-w_{i}=0,$ let $\mathbb{R}%
\begin{pmatrix}
a \\ 
b%
\end{pmatrix}%
$ (for a primitive vector $v=
\begin{pmatrix}
a \\ 
b%
\end{pmatrix}%
\in \mathbb{Z}^{2}$) be a line not containing any $v_i$.
Take $v^{\prime }=%
\begin{pmatrix}
c \\ 
d%
\end{pmatrix}%
$ be another primitive element such that $\{v,v^{\prime }\}$ is a basis for $%
\mathbb{Z}^{2}.$ Let 
\begin{equation*}
\mathrm{proj}:\mathbb{Z}^{2}\hookrightarrow \mathbb{R}^{2}\rightarrow \mathbb{Z}
\end{equation*}%
the projection onto the component $\mathbb{R}v^{\prime }.$ In this case, we
have nonzero $\mathrm{proj}(v_{i}),i=1,2,...,k.$ A similar argument finishes
the proof.

Conversely, if for some $i\neq j,$ the two vectors $v_{i}-w_{i}$, $%
v_{j}-w_{j}$ are not parallel, then $\mathrm{Span}_{\mathbb{Z}%
}\{v_{i}-w_{i},v_{j}-w_{j}\}$ is a finite index subgroup of $\mathbb{Z}^{2}.$
For any group homomorphism $f:\mathbb{Z}^{2}\rightarrow \mathbb{Z}$ with $%
f(v_{i})=f(w_{i}),f(v_{j})=f(w_{j}),$ it has $\mathrm{Span}_{\mathbb{Z}%
}\{v_{i}-w_{i},v_{j}-w_{j}\}\subset \ker f,$ implying $f$ is actually
trivial. This implies $f(v_{i})=0.$ Therefore, $G$ cannot have a
homomorphism to $\mathbb{Z}$ which is non-zero on every edge group.

If $v_{1}-w_{1}$, $v_{2}-w_{2},...,v_{k}-w_{k}$ lie in a common line $l$,
but $l$ contains some vector $v_{i}$ (or $w_{i}$), any group homomorphism $f:%
\mathbb{Z}^{2}\rightarrow \mathbb{Z}$ with $f(v_{j})=f(w_{j}),j=1,2,...,k,$
has $f(v_{i})=0.$ Again, $G$ cannot have a homomorphism to $\mathbb{Z}$
which is non-zero on every edge group.
\end{proof}

We now prove the following proposition.

\begin{proposition}
\label{th10} If a tubular
group is  free-by-cyclic  (resp. virtually free-by-cyclic), then it is  $F_{n}$-by-$\mathbb{Z}$ (resp. virtually $F_{n}$-by-$\mathbb{Z}$).
\end{proposition}

To prove Proposition \ref{th10}, we need the following lemma.

\begin{lemma}
\label{leml2}Suppose that $G$ is finitely generated free-by-cyclic. The
group $G$ is $F_{n}$-by-$\mathbb{Z}$ if and only if the first $L^{2}$-Betti
number $b_{1}^{(2)}(G)=0.$
\end{lemma}

\begin{proof}
Suppose that $G=F_{n}\rtimes \mathbb{Z}$. A result of L\"{u}ck \cite[Theorem 2.1]{luck}
 implies that a mapping torus of a finite CW complex has its
first $L^{2}$-Betti number vanishing. This gives that $b_{1}^{(2)}(G)=0.$

Conversely, let $G$ be a finitely generated free-by-cyclic group with $%
b_{1}^{(2)}(G)=0.$ Let us recall \cite[Proposition 2.3]{gkl} of Gardam--Kielak--Logan. Let $G$ be a torsion-free group that satisfies the Atiyah
conjecture. Let $\phi:G\rightarrow \mathbb{Z}$ be an epimorphism whose
kernel $K$ is a free product of finitely generated groups. If the first $%
L^{2}$-Betti number of $G$ is equal to zero, then $K$ is finitely generated.
In our situation, since $G$ is free-by-cyclic, it satisfies the Atiyah
conjecture (by the work of Linnell \cite{Lin}). Therefore, $K=\ker \phi $
is finitely generated.
\end{proof}

\begin{proof}[Proof of Proposition \protect\ref{th10}]
Chatterji--Hughes--Kropholler \cite[Theorem 1.5]{chk}  prove that when a group 
$G$ acts on a tree $T$ with $V,E$ as the $G$-orbits of vertices and edges$,$
we have 
\begin{equation*}
b_{1}^{(2)}{(G)}=\sum_{v\in V}(b_{1}^{(2)}(G_{v})-\frac{1}{|G_{v}|})+\sum_{e\in E}%
\frac{1}{|G_{e}|},
\end{equation*}%
if $b_{1}^{(2)}(G_{e})=0$ for each $e\in E.$ Let $G$ be a tubular group,
which is a finite graph of groups with vertex group $\mathbb{Z}^{2}$ and
edge group $\mathbb{Z}.$ The group $G$ acts on its Bass-Serre tree. Since $%
\mathbb{Z}^{2}$ has vanishing $L^{2}$-Betti numbers, we know that $%
b_{1}^{(2)}(G)=0$ by the above-mentioned result of
Chatterji--Hughes--Kropholler. Lemma \ref{leml2} implies that a
tubular group $G$ is free-by-cyclic if and only if it is $F_{n}$-by-$\mathbb{%
Z}$. Note that a finite index subgroup of a tubular group is tubular. The
case of virtually free-by-cyclic can be proved similarly. 
\end{proof}

\begin{remark}
    Here is a proof of Proposition \protect\ref{th10} using deficiency which we learned  from an anonymous referee. Note first that by induction one can prove that the deficiency of a tubular group $G$ is at most $1$ using the presentation complex is already aspherical. In fact if the presentation complex is already aspherical, then the deficiency of the group is the deficiency of this presentation. Suppose now that $G$ is a finitely generated $F\rtimes \mathbb{Z}$ group for some free group $F$. If $F$ is not finitely generated, by \cite[Proposition 2.3]{FH99}, $G$ must  have deficiency at least $2$. In fact, if $G$ has deficiency $1$, \cite[Proposition 2.3]{FH99} implies that $G$ is a strict ascending HNN extension of a finitely generated free group and $F$ is the kernel of this extension. But then this is a contradiction as the kernel is not free. 
\end{remark}

\section{Virtual specialness}

In this section, we study the virtual specialness of some tubular free-by-cyclic groups. 

Let us first recall some results related to actions of tubular groups on CAT (0) cube complexes, see \cite{wi14, Wood2, Wood} for more information. Wise obtained in \cite{wi14} a criterion for a tubular group to act freely on
a (possibly infinite dimensional) $\mathrm{CAT(0)}$ cube complex.
Suppose that a tubular group $G$ is a group that splits as a finite graph $\Gamma$ of groups where each vertex group $G_v$ is isomorphic to $\mathbb{Z}^2$ and each edge group $G_e$ is isomorphic to $\mathbb{Z}$. For each edge $e$ of $\Gamma$, let $v(e_{\rightarrow })$ (resp. $v(e_{\leftarrow })$) denote the initial vertex (resp. the terminal vertex) and $e_{\rightarrow }$ (resp. $e_{\leftarrow}$) denote the generator of the edge
subgroup in $G_{v(e_{\rightarrow })}$ (resp. $G_{v(e_{\leftarrow })}$). For elements $s,v \in \mathbb{Z}^2$, let $\langle
s,v\rangle$ be the subgroup generated by $s,v$. Note that for $$s=%
\begin{bmatrix}
x \\ 
y%
\end{bmatrix}%
,v=%
\begin{bmatrix}
x^{\prime } \\ 
y^{\prime }%
\end{bmatrix}%
$$ that are linearly indepedent,  the index $[\mathbb{Z}^{2},\langle s,v\rangle ]=|\det [s,v]|=\#[s,v]$, which is the geometric intersection number of curves in the torus $T^2$ (with $\pi_1(T^2)=\mathbb{Z}^2$) represented by $s,v$. Abusing notation, we also use $s$ to denote the corresponding closed curve in $T^2$.

\begin{lemma}\cite[Theorem~1.1]{wi14}
A tubular group $G(\Gamma )$ acts freely on a $\mathrm{CAT(0)}$ cube complex if and only if
there exists a finite subset $S_{v}\subseteq G_{v}=\mathbb{Z}^{2}$ for each
vertex $v$ satisfying

1) For each edge $e$,
\begin{equation*}
\sum_{s\in S_{v(e_{\rightarrow })}}|[\mathbb{Z}^{2},\langle s,e_{\rightarrow
}\rangle ]|=\sum_{s\in S_{v(e_{\leftarrow })}}|[\mathbb{Z}^{2},\langle
s,e_{\leftarrow }\rangle ]|;
\end{equation*}
 
2) $\langle S_{v}\rangle $ is of finite index in $G_{v}.$
\end{lemma}

Such a collection of finite sets $\{
S_{v}\}_{v\in \Gamma^0}$ is called an equitable set. The tubular group $G(\Gamma)$ is the fundamental group of a graph of spaces $X(\Gamma)$, with vertex space a torus $X_v=S^1 \times S^1$ for each vertex $v$ and edge space $X_e=S^1$ for each edge $e$ of $\Gamma$. The space $X(\Gamma)$ is obtained from the union of $X_v=S^1 \times S^1$ by attaching cylinders $X_e \times [-1,1]$. An immersed wall is either of the form $X_e \times 0$ 
 or a connected component of a graph built from the disjoint union of all the circles in $S_v=\{ \sigma_{vj} : v \in \Gamma^0, 1 \leq j \leq |S_v| \}$  together with some connecting arcs  joining points in $e_{\leftarrow} \cap \sigma_{v(e_{\leftarrow}),i} $ bijectively with points in $e_{\rightarrow} \cap \sigma_{v(e_{\rightarrow}),i'}$ for each edge $e$. Here one uses the fact that $\{
S_{v}\}_{v\in \Gamma^0}$ is equitable, so the number of intersections is the same.

Woodhouse shows in \cite[Theorem 1.1]{Wood} that a tubular group $G$ is virtually special if and only if it acts freely on a finite-dimensional $\mathrm{CAT(0)}$ cube complex. He also proves in \cite[Theorem 1.2]{Wood2}  that the dual cube complex $C(\tilde{X},  W)$ (the wall space defined by Wise \cite{wi14}) is finite dimensional if and only if each immersed wall is non-dilated (as described in the following). 

We now define the dilation function $R$ associated to an immersed wall. Let $\Lambda $ be an immersed wall (not consisting of a single immersed circle) and let $q:\Lambda \rightarrow \Omega $
be the quotient map obtained by quotienting each circle in the equitable set $S_v$
to a vertex. The dilation function $R:\pi _{1}(\Lambda )\rightarrow \mathbb{Q%
}^{\ast }$ factors through $q_{\ast }:\pi _{1}(\Lambda )\rightarrow \pi
_{1}(\Omega )$ \cite[\S5]{Wood2}:
\begin{equation*}
\begin{array}{ccc}
\pi _{1}(\Lambda ) & \overset{R}{\rightarrow } & \mathbb{Q}^{\ast } \\ 
q_{\ast }\downarrow & \nearrow R_{\ast } &  \\ 
\pi _{1}(\Omega ) &  & 
\end{array}%
\end{equation*}%
Given an orientation to each edge in $\Omega $ such that all edges in the
same edge space have the same orientations. Let $\alpha$ be a connecting arc of $\Lambda$ contained in the edge cylinder
$X_e$. Let $C_{\leftarrow}$ and $C_{\rightarrow}$ be the circles containing
the initial and terminal endpoints of $\alpha$, respectively.  Let $\alpha
_{\leftarrow},\alpha _{\rightarrow}$ be the corresponding elements of $C_{\leftarrow},C_{\rightarrow}$ in the
equitable set. Define%
\begin{equation*}
\omega (\alpha )=\frac{|\det [e_{\rightarrow },\alpha _{\rightarrow}]|}{|\det
[e_{\leftarrow },\alpha_{\leftarrow}]|}.
\end{equation*}%
If $\gamma =\sigma _{1}^{\varepsilon _{1}}\sigma _{2}^{\varepsilon
_{2}}\cdots \sigma _{n}^{\varepsilon _{n}}$ is an edge path in $\Omega $
where $\sigma _{i}$ is an oriented arc in $\Omega $, and $\varepsilon
_{i}=\pm 1,$ then define $\omega (\gamma )=\omega (\sigma _{1})^{\varepsilon
_{1}}\omega (\sigma _{2})^{\varepsilon _{2}}\cdots \omega (\sigma
_{n})^{\varepsilon _{n}}$. Moreover, when $\alpha \in \pi_1(\Omega)$, we declare that $R^\ast(\gamma):= \omega(\gamma)$. The immersed wall $\Lambda $ is called
dilated if $\mathrm{Im}R=\mathrm{Im}R_{\ast }$ is infinite. Otherwise, an immersed wall is called non-dilated (see Woodhouse \cite[Definition 4.5]{Wood2}).

Theorem \ref{main} (3) is the following.
\begin{theorem}
\label{th3}When $\{v_{1},w_{1}\}$ are linearly independent and 
\begin{eqnarray*}
\det [w_{1}-v_{1},v_{i}] &=&\pm \det [w_{1}-v_{1},w_{i}], \\
\det [w_{1}+v_{1},v_{i}] &=&\pm \det [w_{1}+v_{1},w_{i}],i=2,...,k,
\end{eqnarray*}%
the group $G=\langle \mathbb{Z}%
^{2},s_{i}|s_{i}v_{i}s_{i}^{-1}=w_{i},i=1,2,...,k\rangle $ is virtually
special.
\end{theorem}

\begin{proof}
Choose 
\begin{equation*}
S=\{z_{1}:=w_{1}-v_{1},z_{2}:=w_{1}+v_{1}\}.
\end{equation*}%
By Woodhouse \cite[Theorem 1.1]{Wood}, it is enough to prove that $G$ acts freely
on a finite-dimensional $\mathrm{CAT(0)}$ cube complex. According to \cite[Theorem 1.2]{Wood2}, it is enough to show that each immersed wall is non-dilated (so that the wall space $C(\tilde{X},  W)$ is finite dimensional). 
Note that $S$ consists of two elements and an immersed wall is supported on the two
elements. For each stable letter $s_i$, denote by $s_i^{-},s_i^{+}$ the two curves in the vertex space represented by $v_i,w_i$. The intersection numbers are (assuming $i\geq 2$)%
\begin{eqnarray*}
\#[z_{1},s_{1}^{-}] &=&|\det [w_{1}-v_{1},v_{1}]|=|\det [w_{1},v_{1}]|=|\det
[w_{1}-v_{1},w_{1}]=\#[z_{1},s_{1}^{+}], \\
\#[z_{2},s_{1}^{-}] &=&|\det [w_{1}+v_{1},v_{1}]|=|\det [w_{1},v_{1}]|=|\det
[w_{1}+v_{1},w_{1}]=\#[z_{2},s_{1}^{+}], \\
\#[z_{1},s_{i}^{-}] &=&|\det [w_{1}-v_{1},v_{i}]|=|\det
[w_{1}-v_{1},w_{i}]|=\#[z_{1},s_{i}^{+}], \\
\#[z_{2},s_{i}^{-}] &=&|\det [w_{1}+v_{1},v_{i}]|=|\det
[w_{1}+v_{1},w_{i}]|=\#[z_{2},s_{i}^{+}].
\end{eqnarray*}%
Since the two ends of any edge space have the same intersection numbers
with $z_{i},i=1,2$, we can choose immersed walls $\Lambda $ connecting the
intersection points in $z_{i}\cap e_{s_{i}\rightarrow }$ to $z_{i}\cap
e_{s_{i}\leftarrow }$ as in Figure \ref{fig:graph}.

\begin{figure}
    \centering
    \includegraphics[width=0.8\linewidth]{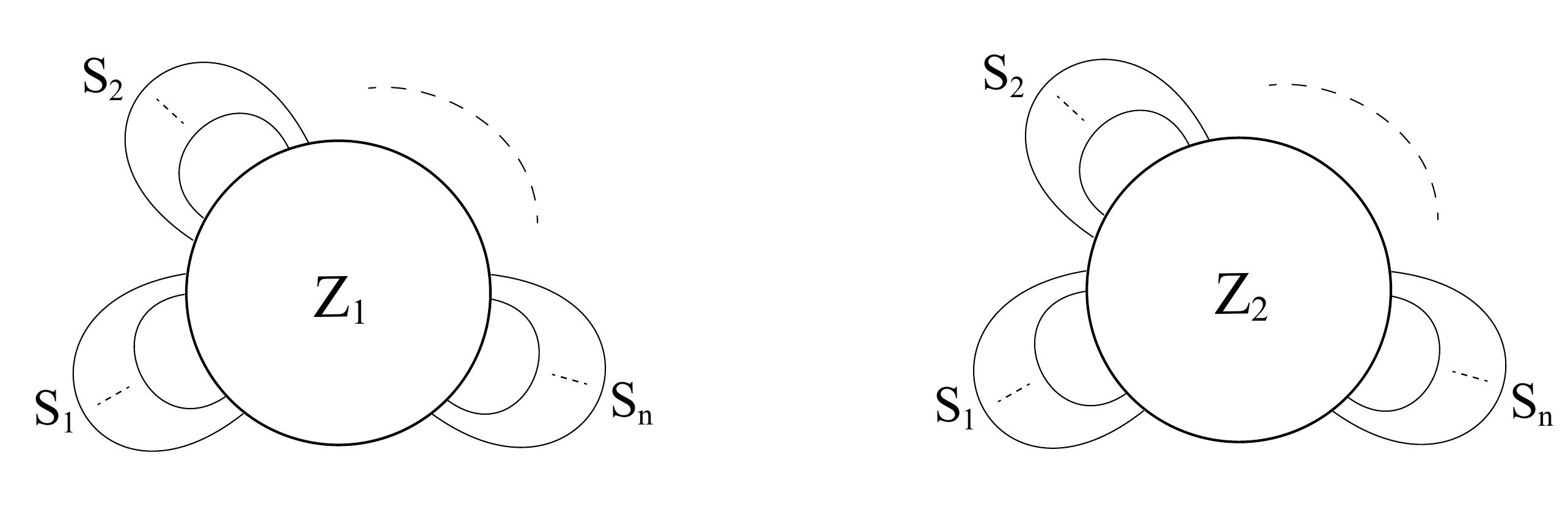}
    \caption{Non-dilated immersed walls. Here the equitable set consists of two curves $z_1,z_2$. Each immersed wall connects $z_i, i=1,2,$ to itself.}
    \label{fig:graph}
     
\end{figure}

Let $\Omega $ be the quotient obtained from $\Lambda $ by quotient each
circle in the equitable set to a vertex. Note that each closed loop in the
quotient $\Omega $ is a composition of edges labeled by $s_{1},$ $%
s_{i},i=2,...,k.$ But each single edge $e_{s_{1}},e_{s_{i}}$ labeled by $%
s_{1},s_{i}$ has the weight 
\begin{eqnarray*}
w(e_{s_{1}}) &=&\frac{\#[z_{1},e_{s_{1}}^{-}]}{\#[z_{1},e_{s_{1}}^{+}]}=1=%
\frac{\#[z_{2},e_{s_{1}}^{-}]}{\#[z_{2},e_{s_{1}}^{+}]}, \\
w(e_{s_{i}}) &=&\frac{\#[z_{1},e_{s_i}^{-}]}{\#[z_{1},e_{s_i}^{+}]}=1=\frac{%
\#[z_{2},e_{s_{i}}^{-}]}{\#[z_{2},e_{s_{i}}^{+}]}.
\end{eqnarray*}%
Therefore, each immersed wall is non-dilated and the group $G$ is virtually
special. When there is
only one vertex space, it is also noted by Wise \cite[Remark 3.6]{wi14} that it is enough to prove the geometric intersection numbers $\#[z_{1},s_{i}^{-}]=%
\#[z_{1},s_{i}^{+}]$ for any $i=1,2,...,k$ to have a finite-dimensional wall space, though our proof is independent of this.
\end{proof}

\begin{corollary}
Suppose that the group $G=\langle \mathbb{Z}%
^{2},s_{i}|s_{i}v_{i}s_{i}^{-1}=w_{i},i=1,2\rangle $ is free-by-cyclic. When
$\{v_{1},w_{1}\}$ are linearly independent and $\det [v_{1},v_{2}]=\det
[v_{1},w_{2}],$ the group $G$ is virtually special.
\end{corollary}

\begin{proof}
When $G$ is free-by-cyclic, we know that $w_{1}-v_{1}$ is parallel to $%
w_{2}-v_{2}$ by Theorem \ref{th2}. This implies that $\det [w_{1}-v_{1},w_{2}-v_{2}]=0$ and $\det
[w_{1}-v_{1},v_{2}]=\det [w_{1}-v_{1},w_{2}].$ Now 
\begin{eqnarray*}
\det [w_{1}+v_{1},v_{2}] &=&\det [w_{1}-v_{1},v_{2}]+\det
[2v_{1},v_{2}], \\
\det [w_{1}+v_{1},w_{2}]&=&\det [w_{1}-v_{1},w_{2}]+\det [2v_{1},w_{2}].
\end{eqnarray*}
When $\det [v_{1},v_{2}]=\det [v_{1},w_{2}],$ Theorem \ref{th3} finishes
the proof.
\end{proof}

Now we consider $\mathrm{CAT(0)}$ free-by-cyclic tubular groups.

\begin{lemma}
\label{lemcat}Let $G=\langle \mathbb{Z}%
^{2},s_{i}|s_{i}v_{i}s_{i}^{-1}=w_{i},i=1,2,...,k\rangle $ be a multiple HNN
extension of the $\mathbb{Z}^{2}$. Suppose that%
\begin{equation*}
v_{1}=%
\begin{pmatrix}
k_{1} \\ 
0%
\end{pmatrix}%
,w_{1}=%
\begin{pmatrix}
0 \\ 
k_{2}%
\end{pmatrix}%
\end{equation*}%
for some nonzero integers $k_{1},k_{2}$ and 
\begin{equation*}
v_{i}=%
\begin{pmatrix}
a_{i} \\ 
b_{i}%
\end{pmatrix}%
,w_{i}=%
\begin{pmatrix}
c_{i} \\ 
d_{i}%
\end{pmatrix}%
,i=2,...,k.
\end{equation*}%
If $G$ is $\mathrm{CAT(0)}$ and free-by-cyclic, then 
\begin{equation}
-k_2{(a_i-c_i)}
=
k_1(b_i-d_i),
\tag{Eq1}
\end{equation}

 and

\begin{equation*}
a_{i}b_{i}=c_{i}d_{i}
\end{equation*}%
for each $i=2,...,k.$
\end{lemma}

\begin{proof}
When $G$ is free-by-cyclic, Theorem \ref{th2} implies that
\[
v_i-w_i=
\begin{pmatrix}
a_i-c_i\\
b_i-d_i
\end{pmatrix}
\]
is parallel to
\[
v_1-w_1=
\begin{pmatrix}
k_1\\
-k_2
\end{pmatrix}.
\]
Therefore,
\begin{equation}
-k_2{(a_i-c_i)}
=
k_1(b_i-d_i),
\end{equation}
for each $i=2,\ldots,k$.

Set
\[
x_i=\frac{a_i}{k_1},\qquad
y_i=\frac{b_i}{k_2},\qquad
x_i'=\frac{c_i}{k_1},\qquad
y_i'=\frac{d_i}{k_2}.
\]
Equation (Eq1) implies that
\[
x_i+y_i=x_i'+y_i'.
\]
When $G$ is $\mathrm{CAT}(0)$, Theorem \ref{th1} implies that
\[
x_i^2+y_i^2+2x_iy_i\cos\theta
=
(x_i')^2+(y_i')^2+2x_i'y_i'\cos\theta
\]
for some common $\theta\in(0,\pi)$. Rewriting this equality gives
\[
(x_i+y_i)^2-2(1-\cos\theta)x_iy_i
=
(x_i'+y_i')^2-2(1-\cos\theta)x_i'y_i'.
\]
Since $x_i+y_i=x_i'+y_i'$ and $\theta\in(0,\pi)$, we have
$1-\cos\theta\neq 0$. It follows that
\[
x_iy_i=x_i'y_i'.
\]
Consequently,
\[
\frac{a_ib_i}{k_1k_2}
=
\frac{c_id_i}{k_1k_2}.
\]
Since $k_1k_2\neq 0$, we conclude that
\[
a_ib_i=c_id_i
\]
for each $i=2,\ldots,k$.
\end{proof}

\begin{theorem}\label{thm:fr-cyc-vspel}
Let $G=\langle \mathbb{Z}^{2},s_{i}|s_{i}v_{i}s_{i}^{-1}=w_{i},i=1,2,...,k%
\rangle $ be a $\mathrm{CAT(0)}$ free-by-cyclic tubular group. Then $G$ is virtually
special.
\end{theorem}

\begin{proof}
When $v_{i},w_{i}$ are linearly independent for some $i$, we assume $i=1$
without loss of generality. There are linearly independent primitive
elements $e_{1},e_{2}\in \mathbb{Z}^{2}$ such that $v_{1}=e_{1}^{k_{1}}$ and
$w_{1}=e_{2}^{k_{2}}$ for some positive integers $k_{1},k_{2}.$ Using
$\{e_{1},e_{2}\}$ as a basis of $\mathbb{Q}^{2}$, we write
\begin{equation*}
v_{1}=
\begin{pmatrix}
k_{1} \\
0
\end{pmatrix},
\qquad
w_{1}=
\begin{pmatrix}
0 \\
k_{2}
\end{pmatrix},
\end{equation*}
and
\begin{equation*}
v_{i}=
\begin{pmatrix}
a_{i} \\
b_{i}
\end{pmatrix},
\qquad
w_{i}=
\begin{pmatrix}
c_{i} \\
d_{i}
\end{pmatrix},
\qquad i=2,\ldots,k.
\end{equation*}
Here $a_i,b_i,c_i,d_i\in\mathbb{Q}$. The proof of Lemma \ref{lemcat}
applies verbatim over $\mathbb{Q}$. Moreover, this change of coordinates
multiplies every determinant below by the same nonzero factor, so it is
enough to verify the conditions of Theorem \ref{th3} in these coordinates.

By Lemma \ref{lemcat}, we have
\begin{equation}
a_{i}b_{i}=c_{i}d_{i} \tag{Eq2}
\end{equation}
for each $i=2,\ldots,k.$ In order to prove that $G$ is virtually special,
it is enough to check the conditions of Theorem \ref{th3}. It is direct that
\begin{eqnarray*}
\det [w_{1}-v_{1},v_{i}]
&=&
\det
\begin{bmatrix}
-k_{1} & a_{i} \\
k_{2} & b_{i}
\end{bmatrix}
=-k_{1}b_{i}-k_{2}a_{i}, \\
\det [w_{1}-v_{1},w_{i}]
&=&
\det
\begin{bmatrix}
-k_{1} & c_{i} \\
k_{2} & d_{i}
\end{bmatrix}
=-k_{1}d_{i}-k_{2}c_{i}.
\end{eqnarray*}
By the identity (Eq1),
\[
-k_{2}(a_{i}-c_{i})=k_{1}(b_{i}-d_{i}),
\]
which implies that
\[
-k_{1}b_{i}-k_{2}a_{i}
=
-k_{1}d_{i}-k_{2}c_{i}.
\]

It remains to prove
\[
\det [w_{1}+v_{1},v_{i}]
=
\pm\det [w_{1}+v_{1},w_{i}],
\]
which is equivalent to
\begin{equation}
k_{1}b_{i}-k_{2}a_{i}
=
\pm(k_{1}d_{i}-k_{2}c_{i}). \tag{Eq3}
\end{equation}
Set
\[
P=k_{1}b_{i},\qquad Q=k_{2}a_{i},\qquad
P'=k_{1}d_{i},\qquad Q'=k_{2}c_{i}.
\]
Equation (Eq1) gives
\[
P+Q=P'+Q',
\]
while Equation (Eq2) gives
\[
PQ=k_{1}k_{2}a_{i}b_{i}
   =k_{1}k_{2}c_{i}d_{i}
   =P'Q'.
\]
Consequently,
\[
(P-Q)^{2}
=(P+Q)^{2}-4PQ
=(P'+Q')^{2}-4P'Q'
=(P'-Q')^{2}.
\]
Therefore,
\[
P-Q=\pm(P'-Q'),
\]
which is precisely (Eq3). Hence the conditions of Theorem \ref{th3} are
satisfied, and $G$ is virtually special.

Suppose now that $v_{i}$ and $w_{i}$ are parallel for every
$i=1,2,\ldots,k.$ There are primitive elements
$u_{i}\in\mathbb{Z}^{2}$ and integers $m_{i},n_{i}$ such that
\[
v_{i}=u_{i}^{m_{i}},
\qquad
w_{i}=u_{i}^{n_{i}}.
\]
Since $G$ is $\mathrm{CAT(0)}$, we have $|m_{i}|=|n_{i}|$ for every
$i=1,2,\ldots,k.$ Choose a primitive element $z$ such that
$\det[u_{1},z]=1.$ Consider the graph of spaces associated to the tubular
group $G.$ Then the set $S=\{u_{1},z\}$ generates $\mathbb{Z}^{2}$ and,
for every $s\in S$ and every $i=1,2,\ldots,k$, we have
\[
\#[s,v_i]
=
|m_i|\,|\det[s,u_i]|
=
|n_i|\,|\det[s,u_i]|
=
\#[s,w_i].
\]
Thus $S$ is an equitable set. Pairing the intersection points belonging
to the same element $s\in S$, every connecting arc has weight $1$.
Therefore, every immersed wall is non-dilated. Woodhouse's criterion
\cite[Theorem 1.1]{Wood} now implies that $G$ is virtually special, as in
the proof of Theorem \ref{th3}.
\end{proof}

\begin{corollary}
\label{cor1}Let $G=\langle \mathbb{Z}%
^{2},s_{i}|s_{i}v_{i}s_{i}^{-1}=w_{i},i=1,2,...,k\rangle $ be a
free-by-cyclic tubular group. Then $G$ is $\mathrm{CAT(0)}$ if and only if $G$ is
virtually special.
\end{corollary}

\begin{proof}
Wise proves that a virtually special tubular group is $\mathrm{CAT(0)}$ (cf.
\cite[Lemma 4.4]{wi14}). Theorem \ref{thm:fr-cyc-vspel} proves the other direction.
\end{proof}

 Two edge groups incident at a vertex $v$ are parallel if their images in $G_v\cong \mathbb{Z}^2$ are commensurable; equivalently, their nonzero generating vectors span the same one-dimensional subspace of $G_v \otimes \mathbb{Q}$. A parallelism class is an equivalence class under this relation. The following lemma improves Wise's result \cite[Corollary 5.10]{wi14} to the virtual setting.
\begin{lemma}\label{lem:vir-cocom--cube}
    Let $G$ be a tubular group. 
    \begin{enumerate}
        \item Suppose that there are more than two parallelism classes of edge groups in some vertex group $G_v$. Then $G$ cannot virtually act freely and cocompactly on a $\mathrm{CAT(0)}$ cube complex.
        \item Suppose that $G$ has no Baumslag--Solitar subgroup $\langle a,t\mid ta^mt^{-1} =a^n\rangle $ where $m\neq \pm n$. Then the following are equivalent:
        \begin{enumerate}[label=(\roman*)]
            \item $G$  acts freely and cocompactly on a $\mathrm{CAT(0)}$ cube complex.
            \item $G$  virtually acts freely and cocompactly on a $\mathrm{CAT(0)}$ cube complex.
            \item There are exactly one or two parallelism classes of edge groups in each vertex group.
        \end{enumerate}
    \end{enumerate}
\end{lemma}
\begin{proof}
Let us prove (1) first. By \cite[Corollary 5.10]{wi14}, $G$ cannot act freely and cocompactly on a $\mathrm{CAT(0)}$ cube
complex. Let   $T$ be the corresponding Bass--Serre tree of $G$ and $H$ be any finite index subgroup of $G$. Then the finite-index subgroup $H$ also acts on $T$ cocompactly with $\mathbb{Z}^2$ as vertex groups and $\mathbb{Z}$ as edge groups, hence it is a tubular group. The corresponding Bass--Serre tree of $H$ is again $T$. Suppose $G_v$ is a vertex group in $G$ that has at least three parallelism classes of edge groups, say $G_{e_1}, G_{e_2}$ and $G_{e_3}$. View $[G_v] \in G/G_v$ as a vertex in $T$. Then $G_{e_1}, G_{e_2}$ and $G_{e_3}$ represent the stabilizers of three edges that connect to $G_v$. Consider now the action of $H$ on $T$. The stabilizer of $[G_v]$ under the action of $H$ is the finite index subgroup $G_v \cap H$, and the corresponding stabilizer $H_{e_i}$ of  $e_i$ is the finite index subgroup $G_{e_i}\cap H$, $i=1,2,3$. Since $G_{e_1}, G_{e_2}$ and $G_{e_3}$ represent $3$ parallelism classes of edge groups in the vertex group $G_v$, we must also have $H_{e_1}$, $H_{e_2}$, $H_{e_3}$ represent $3$ parallelism classes of edge groups in $H_v$.   Applying \cite[Corollary 5.10]{wi14} again, we have that $H$ cannot act freely and cocompactly on a $\mathrm{CAT(0)}$ cube complex. 

For (2), $(i) $ implies $(ii)$ trivially, while $(ii)$ implies $(iii)$ following from $(1)$ and $(iii)$ implies $(i)$ following from \cite[Corollary 5.10]{wi14}. 

\end{proof}

\begin{corollary}
\label{cor2}Let $G=\langle \mathbb{Z}%
^{2},s_{i}|s_{i}v_{i}s_{i}^{-1}=w_{i},i=1,2\rangle .$ Suppose that $%
\{v_{1},w_{1}\}$ are linearly independent, and $[v_{1},w_{1}]^{-1}v_{2}=%
\begin{pmatrix}
x \\ 
y%
\end{pmatrix}%
,[v_{1},w_{1}]^{-1}w_{2}=%
\begin{pmatrix}
x^{\prime } \\ 
y^{\prime }%
\end{pmatrix}%
.$ Suppose that

\begin{enumerate}
\item $
x^{2}+y^{2}+2xy\cos \theta ={x^{\prime }}^{2}+{y^{\prime
}}^{2}+2x^{\prime }y^{\prime }\cos \theta
$
for some $\theta \in (0,\pi )$;

\item  $v_{1}-w_{1}$, $v_{2}-w_{2}$ lie in a line $l$ which doesn't contain any
element of $\{v_{1},v_{2},w_{1},w_{2}\}$;

\item $v_{2}$ or $w_{2}$ does not lie in the union of the two lines $\mathbb{R}v_{1}\cup \mathbb{R}%
w_{1}.$
\end{enumerate}
Then $G$ is a $\mathrm{CAT(0)}$, free-by-cyclic,  virtually non-cocompact special group; indeed, it is not virtually cocompactly cubulated.
\end{corollary}

\begin{proof}
According to Theorem \ref{main}, Conditions 1) and 2) imply that
the group $G$ is $\mathrm{CAT(0)}$ and free-by-cyclic. Theorem \ref{thm:fr-cyc-vspel} implies that it is virtually special. Since $v_1$  and $w_1$ are linearly independent, their edge groups determine two distinct parallelism classes. Condition (3) supplies a third parallelism class. Lemma 4.7(1) therefore implies that no finite-index subgroup of $G$ acts freely and cocompactly on a CAT(0) cube complex.

By \Cref{lem:vir-cocom--cube} (1),  Condition (3) implies that $G$ cannot be virtually cocompactly 
special. 
\end{proof}

Let now $G(\{p_i,q_i\}_{i=1}^{n-1})=F_{n}\rtimes _{\phi }\mathbb{Z}$ be as  defined in the introduction. 

\begin{theorem}
\label{th7-1}The group $G(\{p_i,q_i\}_{i=1}^{n-1})$ is virtually
special if and only if either

1) $p_{i}=-q_{i}$ for each $i=1,2,...,n-1$, or

2) $p_{s}\neq -q_{s}$ for some $s$ and 
\begin{equation*}
q_{i}(q_{i}+p_{s}-q_{s})=p_{i}(p_{i}-p_{s}+q_{s})
\end{equation*}%
for each $i=1,2,...,n-1.$
\end{theorem}

\begin{proof}
Note that $G(\{p_i,q_i\}_{i=1}^{n-1})$ is a tubular group
\begin{equation*}
G(\{p_i,q_i\}_{i=1}^{n-1})
=\langle \mathbb{Z}^{2},s_i
\mid s_iv_is_i^{-1}=w_i,\ i=1,2,\ldots,n-1\rangle
\end{equation*}
with
\begin{equation*}
v_i=
\begin{pmatrix}
q_i\\
1
\end{pmatrix},
\qquad
w_i=
\begin{pmatrix}
-p_i\\
1
\end{pmatrix},
\qquad i=1,2,\ldots,n-1.
\end{equation*}
By Corollary \ref{cor1}, the group is virtually special if and only if it is
$\mathrm{CAT}(0)$. If $p_i=-q_i$ for each $i=1,2,\ldots,n-1$, it is obvious
that the multiple HNN extension $G(\{p_i,q_i\}_{i=1}^{n-1})$ is
$\mathrm{CAT}(0)$. When $p_s\neq-q_s$ for some $s$, without loss of generality
we assume that $s=1$. The two vectors $v_1,w_1$ are linearly independent.
Let $[v_1,w_1]$ be the matrix with columns $v_1,w_1$. Note that
\begin{eqnarray*}
[v_1,w_1]^{-1}v_i
&=&
\frac{1}{p_1+q_1}
\begin{pmatrix}
p_1+q_i\\
q_1-q_i
\end{pmatrix},
\\
{[v_1,w_1]}^{-1}w_i
&=&
\frac{1}{p_1+q_1}
\begin{pmatrix}
p_1-p_i\\
q_1+p_i
\end{pmatrix},
\qquad i=2,\ldots,n-1.
\end{eqnarray*}
Theorem \ref{th1} implies that $G(\{p_i,q_i\}_{i=1}^{n-1})$ is
$\mathrm{CAT}(0)$ if and only if there is a $\theta\in(0,\pi)$ such that
\begin{eqnarray*}
&&(p_1+q_i)^2+(q_1-q_i)^2
 +2(p_1+q_i)(q_1-q_i)\cos\theta\\
&=&(p_1-p_i)^2+(q_1+p_i)^2
 +2(p_1-p_i)(q_1+p_i)\cos\theta
\end{eqnarray*}
for each $i=2,\ldots,n-1$. Since
\begin{equation*}
(p_1+q_i)+(q_1-q_i)
=
(p_1-p_i)+(q_1+p_i)
=
p_1+q_1,
\end{equation*}
the preceding equality is equivalent to
\begin{equation*}
2(1-\cos\theta)
\bigl(
q_i(q_i+p_1-q_1)-p_i(p_i-p_1+q_1)
\bigr)=0.
\end{equation*}
Since $\theta\in(0,\pi)$, we have $1-\cos\theta\neq0$. Therefore,
$G(\{p_i,q_i\}_{i=1}^{n-1})$ is $\mathrm{CAT}(0)$ if and only if
\begin{equation*}
q_i(q_i+p_1-q_1)
=
p_i(p_i-p_1+q_1)
\end{equation*}
for each $i=2,\ldots,n-1$. The equality also holds automatically when $i=1$,
which completes the proof.
\end{proof}
\begin{corollary}\label{cor:vir-cpspe}
The group $G(\{p_i,q_i\}_{i=1}^{n-1})$ is virtually compact special if
and only if the set $\{-p_{i},q_{i}:i=1,2,...,n-1\}$ consists of at most two
elements and one of the two conditions of Theorem \ref{th7-1} holds.
\end{corollary}

\begin{proof}
Wise proves that a compact cubulated tubular group is virtually compact
special \cite[Corollary 5.9]{wi14}. Moreover, by \Cref{lem:vir-cocom--cube} a tubular group $G$ is virtually
compact cubulated if and only if there is exactly one or two parallelism
classes of edge groups in each vertex group and $G$ contains no
Baumslag--Solitar group $BS(m,n)=\langle a,t:ta^{n}t^{-1}=a^{m}\rangle $ with 
$n\neq \pm m.$ Note that the group has a presentation of the form
\begin{equation*}
G(\{p_i,q_i\}_{i=1}^{n-1})=\langle \mathbb{Z}%
^{2},s_{i}|s_{i}v_{i}s_{i}^{-1}=w_{i},i=1,2,...,n-1\rangle
\end{equation*}%
with 
\begin{equation*}
v_{i}=%
\begin{pmatrix}
q_{i} \\ 
1%
\end{pmatrix}%
,w_{i}=%
\begin{pmatrix}
-p_{i} \\ 
1%
\end{pmatrix}%
,i=1,2,...,n-1.
\end{equation*}%
When the tubular group $G(\{p_i,q_i\}_{i=1}^{n-1})$ is virtually
compact special, there are at most two parallelism classes of $v_{i}$ and $%
w_{i},$ implying the set $\{-p_{i},q_{i}:i=1,2,...,n-1\}$ consists of at
most two elements. Since a compact special group is automatically $\mathrm{CAT(0)}$,
one of the two conditions of Theorem \ref{th7-1} holds. Conversely, when the
set $\{-p_{i},q_{i}:i=1,2,...,n-1\}$ consists of at most two elements, there
will be one or two parallelism classes of edge groups. When one of the two
conditions of Theorem \ref{th7-1} holds, the group $G(\{p_i,q_i\}_{i=1}^{n-1})$ is virtually special, hence $\mathrm{CAT(0)}$ by \Cref{cor1}. Since each element in a $\mathrm{CAT(0)}$ group is undistorted, $G(\{p_i,q_i\}_{i=1}^{n-1})$ contains no Baumslag--Solitar group $BS(m,n)$
with $n\neq \pm m.$ Therefore, $G(\{p_i,q_i\}_{i=1}^{n-1})$ is
virtually compact special by  \Cref{lem:vir-cocom--cube} (2).
\end{proof}

The following example shows that the answer to Question \ref{prob2} is negative.
\begin{example}
\label{eg1} Let $\Psi$ be the automorphism of $F_3= \langle a,b,c\rangle$ given by $%
a\rightarrow a,b\rightarrow a^{m}ba^{m},c\rightarrow
a^{n}ca^{n}$ for some integers $m,n$. Let $G_1 = F_3\rtimes_\Psi \mathbb{Z}$ be the corresponding free-by-cyclic group. As shown in Example \ref{exam:Gpq-CAT(0)},  $G_1$ is a $\mathrm{CAT(0)}$ tubular group. By Corollary \ref{cor1}, it is virtually special. On the other hand, it has the following tubular presentation:
$$\langle a,t, b,c\mid [a,t]=1, b^{-1} a^{-m}t b =a^mt, c^{-1} a^{-n}tc = a^nt \rangle.$$
When $|m|\neq |n|$, $a^{-m}t,a^mt, a^{-n}t, a^nt$ forms at least three parallelism
classes of edge groups inside the vertex group generated by $a$ and $t$, so 
$G_1$ is not virtually cocompactly cubulated by \Cref{lem:vir-cocom--cube} (1).

When $|m| = |n|$, $a^{-m}t,a^mt, a^{-n}t, a^nt$ forms at most  two parallelism
classes of edge groups inside the vertex group generated by $a$ and $t$, so by \Cref{lem:vir-cocom--cube} (2), $G_1$ is a CAT(0) cube group. Moreover, it is virtually compact special by \cite[Corollary 5.9]{wi14} or \Cref{cor:vir-cpspe}.
\end{example}

The following free-by-cyclic group $F_3\rtimes_{\Phi} \mathbb{Z}$ is also from Lyman's paper \cite[Introduction]{Ly}.
\begin{example}
    Let $\Phi$ be the automorphism of $F_3= \langle a,b,c\rangle$ given by $a \to a,b\to a^{-1}ba, c\to a^{-2} ca^2$. Let $G_2 = F_3\rtimes_{\Phi} \mathbb{Z}$ be the corresponding free-by-cyclic group. Up to the inner conjugation by $a$, we can replace $\Phi$ by the following automorphism:
    $$a \to a,b\to b, c\to a^{-1} ca.$$
 So $G_2$ has the following tubular presentation:
 $$\langle a,t, b,c\mid [a,t]=1, btb^{-1}=t, c^{-1} atc = at \rangle.$$
    Replace $a$ by $a'=at$, the tubular presentation of $G_2$ can be rewritten as:
$$\langle a',t, b,c\mid [a',t]=1, btb^{-1}=t, c^{-1} a'c = a' \rangle.$$
This is already a $\mathrm{RAAG}$. In particular, $G_2$ is the fundamental group of a compact special $\mathrm{CAT(0)}$ cube complex.
\end{example}

\section{Amalgamation of two free-by-cyclic groups}

In this section, we consider the amalgamation of two (virtually)
free-by-cyclic groups. We prove that the amalgamated product of two 
free-by-cyclic groups along a cyclic
subgroup may not be virtually free-by-cyclic.  Therefore, the answer to  \Cref{prob3} is negative.

 It is well known that the class of (virtually)
free-by-cyclic groups is closed under taking subgroups and free products \cite{bfm},
but not closed under taking HNN extensions along cyclic subgroups. It
was not clear whether the class of free-by-cyclic groups is closed under
amalgams along $\mathbb{Z}$ subgroups (see Hagen--Wise \cite[Remark
3.6]{hw}). Let us first provide a necessary condition on when the amalgamated product is again free-by-cyclic.

We need to introduce some terminology.  Suppose that $G$ is decomposed as a semi-direct
product $H\rtimes \mathbb{\langle }a\mathbb{\rangle }$ for a free subgroup $%
H $ and an element $a.$ The free subgroup $H$ is called a free complement of 
$a.$ An element $x\in G$ is primitive if $x$ is an element of a free basis
of some free complement $H.$ An element $g\in G$ is called a generalized retractor if there
exists a decomposition $G=H\rtimes \mathbb{\langle }a\mathbb{\rangle }$ for a
free subgroup $H$ and a \emph{nontrivial} element $a$ such that $g=ha^{i}$
for some $h\in H$ and a nonzero integer $i$. For example, when $G=\mathbb{Z}$%
, any non-zero element is a generalized retractor and $G$ itself is a free
complement (of the trivial element).

\begin{theorem}
\label{th8}Let $G_{1},G_{2}$ be two groups, $a\in G_{1},b\in
G_{2}$ be two elements of infinite order. Suppose that the amalgamated product $G=G_{1}\ast
_{a=b}G_{2}$ is free-by-cyclic. Then $G_1,G_2$ are  free-by-cyclic, and one of the following holds

\begin{enumerate}
\item[1)] both $a,b$ are generalized retractors;

\item[2)] $a$ and $b$ lie in some free complements of $G_{1},G_{2}$ respectively.
Moreover, $a$ or $b$ is primitive.
\end{enumerate}
\end{theorem}

\begin{proof}
When $G_{1}\ast _{a=b}G_{2}$ is free-by-cyclic, there is an epimorphism 
\begin{equation*}
\phi :G_{1}\ast _{a=b}G_{2}\rightarrow \mathbb{Z}
\end{equation*}%
with a free $\ker \phi .$ Since $\ker \phi |_{G_{i}}\leq\ker \phi ,i=1,2,$ we
know that $\ker \phi |_{G_{i}}$ is free as well. Therefore, both $G_1, G_2$ are free-by-cyclic (noting that a free group is free-by-cyclic).

If $\phi (a)\neq 0\in \mathbb{Z}$, we have $\phi |_{\langle a\rangle }$ is
injective.  Therefore, 
\begin{equation*}
G_{1}=\ker \phi |_{G_{1}}\rtimes \langle x\rangle
\end{equation*}
for a preimage $x$ of the generator of $\mathrm{Im}(\phi |_{G_{1}})$. In this
case, we have $a=hx^{i}$ for the nonzero integer $i$ and an element 
$h\in \ker \phi |_{G_{1}}.$ Therefore, $a$ is a generalized retractor. The
same argument shows that $b$ is a generalized retractor.

If $\phi (a)=0\in \mathbb{Z}$, we have $\phi (b)=0\in \mathbb{Z}$. Since $%
\phi $ is surjective, $\phi |_{G_{1}},\phi |_{G_{2}}$ cannot be trivial
simultaneously. Without loss of generality, we assume that $\phi ({G_{1}}) \cong \mathbb{Z}$. Then 
\begin{equation*}
G_{1}=\ker \phi |_{G_{1}}\rtimes \langle x\rangle
\end{equation*}
for a preimage $x$ of the generator of $\mathrm{Im}(\phi |_{G_{1}})\leq$ $\mathbb{%
Z}$. Consider the action of $\ker \phi $ on the Bass--Serre tree $T$ of the
amalgamated product $G_{1}\ast _{a=b}G_{2}.$ Note that each edge stabilizer
is $\mathbb{Z}$. The group $\ker \phi $ is a graph of free groups amalgamated over
cyclic subgroups. Let $v,w$ be the two vertices in the quotient $T/\ker \phi $, corresponding to $G_1,G_2$. The graph of group $G([v,w])$ of the segment $[v,w]$ has vertex groups $\ker \phi |_{G_1}, \ker \phi |_{G_2}$, and edge group generated by $a=b$. Since $\ker \phi$ is free, the subgroup $G([v,w])$ is free.
  It is well known that $G([v,w])$ is free only if
one of the subgroups $\langle a\rangle,\langle b\rangle $ is primitive ( cf.  
\cite[Theorem 2.1]{bfr}).
\end{proof}

The following shows that condition 1) in Theorem \ref{th8} is also sufficient.

\begin{lemma}
\label{th8-cor}Let $G_{1},G_{2}$ be two free-by-cyclic groups, $a\in G_{1},b\in
G_{2}$ be two elements of infinite order. When  both $a,b$ are generalized retractors, the amalgamated product $G=G_{1}\ast
_{a=b}G_{2}$ is free-by-cyclic. 
\end{lemma}

\begin{proof}
    When both $a,b$ are generalized retractors, there exist
decompositions $G_{1}=H_{1}\rtimes \mathbb{\langle }a_{1}\mathbb{\rangle }$, 
$G_{2}=H_{2}\rtimes \mathbb{\langle }a_{2}\mathbb{\rangle }$ for free
subgroups $H_{1},H_{2}$ and $a=h_{1}a_{1}^{i_{1}},b=h_{2}a_{2}^{i_{2}}$ for
some elements $h_{1}\in H_{1},h_{2}\in H_{2}$ and some nonzero integers $%
i_{1},i_{2}.$ Let 
\begin{equation*}
\phi :G_{1}\ast _{a=b}G_{2}\rightarrow \mathbb{Z}
\end{equation*}
be defined by 
\begin{equation*}
\phi (H_{1})=0,\phi (a_{1})=i_{2},\phi (H_{2})=0,\phi (a_{2})=i_{1}.
\end{equation*}%
It is not hard to see that $\phi $ is a well-defined group homomorphism.
Considering the action of $\ker \phi $ on the Bass--Serre tree of $G_{1}\ast
_{a=b}G_{2},$ the edge stabilizer is trivial and the vertex stabilizers are
free (i.e. conjugates of $H_{1},H_{2}$). Therefore, the kernel $\ker \phi $
is a free product of free groups and thus free.
\end{proof}

\begin{remark}
    When both $G_{1},G_{2}$ are free (then $G$ is called a cyclically pinched
one-relator group),  Condition 1) in Theorem \ref{th8} is known to be sufficient by
Baumslag--Fine--Miller--Troeger \cite[Theorem 2]{bfm}. They also show that
the family of (virtually) free-by-cyclic groups is closed under taking
subgroups and free products.
\end{remark}

\begin{example}
Let $G_{1}$ be a free-by-cyclic group that is not abelian. Let $G_{2}=F_{2}\times \mathbb{Z}$ and $b$ be a generator of
the factor $\mathbb{Z}$. For any nontrivial element $a\in \lbrack
G_{1},G_{1}]$ (the commutator subgroup of $G_1$), the amalgamated product $G_{1}\ast
_{a=b}G_{2}$ is not free-by-cyclic.
\end{example}

\begin{proof}
Note that $b\in G_{2}$ commutes with any element in $F_{2}.$ This means that 
$b$ does not lie in a free complement of any element. Suppose
that $G_{1}\ast _{a=b}G_{2}$ is free-by-cyclic. By Theorem \ref{th8}, the
element $a$ has to be a generalized retractor. However, since $a\in \lbrack
G_{1},G_{1}],$ any surjection $\phi :G_{1}\ast _{a=b}G_{2}\rightarrow 
\mathbb{Z}$ must have $\phi (a)=0.$ This is a contradiction.
\end{proof}

\begin{example}
\label{eg2}Let $G_{1}=\langle a,b,s:ab=ba,sas^{-1}=b\rangle .$ The group $%
G_{1}\ast _{ab^{-1}=a}G_{1}$ is not free-by-cyclic.
\end{example}

\begin{proof}
By Theorem \ref{main} (2), the group $G_{1}$ is free-by-cyclic. It is even $F_n$-by-$\mathbb{Z}$ by \Cref{th10}. The element $%
a $ is a generalized retractor while $ab^{-1}$ is not, as any surjection $%
\phi :G_{1}\rightarrow \mathbb{Z}$ must have $\phi (ab^{-1})=0$. This implies any surjection from $G_{1}\ast _{ab^{-1}=a}G_{1} \to \mathbb{Z}$ must map $a$ in the right copy of $G_1$ to $0$. But since $a,b$ are conjugate, the element $b$ in the right copy of $G_1$ must also map to $0$. Therefore, both $a,b$ lie in the kernel and the kernel cannot be free.    
\end{proof}

The following example shows that the condition (2) in \Cref{th8} is not sufficient. 

\begin{example}
    Let $G_1 =\langle a,s \mid sas^{-1} = a^{-1}\rangle $ and $G_2 = \langle b,t\mid tbt^{-1} =b\rangle $, and $G = G_1 \ast_{a=b} G_2$. By Britton's lemma, or observing that $G$ surjects to the free group generated by $s$ and $t$ (by killing $a$), the elements $s^2$ and $t$ generate a free subgroup of rank $2$ in $G$, and the intersection $\langle a\rangle \cap \langle s^2,t\rangle =\{1\}$. On the other hand, one checks that the element $a$ commutes with every element in $\langle s^2,t\rangle$, implying that elements $a$ and $[s^2,t]$ generate a free abelian subgroup of rank $2$ in $G$.  For any homomorphism $\phi$ from $G$ to $\mathbb{Z}$, we have $\phi(a)=0, \phi([s^2,t]) =0$. But this means that $\ker(\phi)$ contains a free abelian subgroup of rank $2$ and cannot be free. 
\end{example}

For $F_n$-by-$\mathbb{Z}$ groups, we have a similar characterization.

\begin{corollary}
    Let $G_{1},G_{2}$ be two finitely generated groups, $a\in G_{1},b\in
G_{2}$ be two infinite-order elements. If the amalgamated product $G=G_{1}\ast
_{a=b}G_{2}$ is $F_n$-by-$\mathbb{Z}$, then the groups $G_{1},G_{2}$ are $F_n$-by-$\mathbb{Z}$ and either both $a,b$ are generalized retractors, or $a,b$ lie in free complements of $G_{1},G_{2}$ respectively, with $a$ or $b$ is primitive.
\end{corollary}
\begin{proof}
    It is enough to note that the first $L^2$-Betti number $b^{(2)}_1(G)=b^{(2)}_1(G_1)+b^{(2)}_1(G_2)$ by Chatterji--Hughes--Kropholler \cite[Theorem 1.5]{chk}. The finitely generated free-by-cyclic group $G$ is $F_n$-by-$\mathbb{Z}$ if and only if the first $L^2$-Betti number $b^{(2)}_1(G)=0$ by Lemma \ref{leml2}. If $G$ is $F_n$-by-$\mathbb{Z}$, we have $b^{(2)}_1(G_1)=b^{(2)}_1(G_2)=0$ and $G_1,G_2$ are $F_n$-by-$\mathbb{Z}$. The remaining proof can be finished using Theorem \ref{th8}.
\end{proof}

Now we consider the case of virtually free-by-cyclic groups. Let $G$ be a
virtually free-by-cyclic groups. An element $g\in G$ is called a virtual
generalized retractor if there is a finite index torsion-free subgroup $%
\Gamma \leq G$ and a decomposition $\Gamma =N\rtimes \mathbb{\langle }a\mathbb{%
\rangle }$ for a free subgroup $N$ and a nontrivial element $a,$ such that $%
g=na^{i}$ for some $n\in N$ and nonzero integer $i$. Similarly, the subgroup 
$N$ is called a virtual free complement of the element $a$ if there is a
finite index torsion-free subgroup $\Gamma $ decomposed as $\Gamma =N\rtimes 
\mathbb{\langle }a\mathbb{\rangle }$ for a free subgroup $N$ and an element $%
a.$

\begin{theorem}\label{thm:vir-free-nec}
Let $G_{1},G_{2}$ be two finitely generated groups, $a\in
G_{1},b\in G_{2}$ be two infinite-order elements. Suppose that the amalgamated
product $G_{1}\ast _{a=b}G_{2}$ is virtually free-by-cyclic. Then $G_1,G_2$ are virtually free-by-cyclic and there exists a
positive integer $m$ such that one of the following holds

\begin{enumerate}
\item[1)] both $a^{m},b^{m}$ are virtual generalized retractors;

\item[2)] $a^{m}$ and $b^{m}$ lie in virtual free complements of $%
G_{1},G_{2}$ respectively and $a^{m}$ or $b^{m}$ is primitive.
\end{enumerate}
\end{theorem}

\begin{proof}
Suppose that there is a finite index subgroup $\Gamma \leq G_{1}\ast _{a=b}G_{2}$
and an epimorphism $\phi :\Gamma \rightarrow \mathbb{Z}$ with a free $\ker
\phi .$ Note that $G_{1}\cap \Gamma ,G_{2}\cap \Gamma $ are free-by-cyclic
of finite indices in $G_{1},G_{2},$ respectively. Suppose that $\Gamma \cap
\langle a\rangle =\langle a^{m}\rangle $ for a positive integer $m.$
Moreover, we have 
\begin{equation*}
\Gamma \geq(G_{1}\cap \Gamma )\ast _{a^{m}=b^{m}}(G_{2}\cap \Gamma )
\end{equation*}%
by Britton's lemma. Since $\Gamma $ is free-by-cyclic, the subgroup $%
(G_{1}\cap \Gamma )\ast _{a^{m}=b^{m}}(G_{2}\cap \Gamma )$ is free-by-cyclic
as well. Theorem \ref{th8} implies that either both $a^{m},b^{m}$ are
generalized retractors of $G_{1}\cap \Gamma ,G_{2}\cap \Gamma $ respectively,
or $a^{m}$ and $b^{m}$ both lie in free complements of $G_{1}\cap \Gamma
,G_{2}\cap \Gamma $ respectively and $a^{m}$ or $b^{m}$ is primitive.
\end{proof}

\begin{remark}
Some sufficient conditions for getting virtually free-by-cyclic groups are
obtained by Baumslag--Fine--Miller--Troeger \cite[Theorem 4]{bfm}  and
Hagen--Wise \cite{hw}. 
\end{remark}

\begin{definition}
 Let $H\leq G$ be groups. $H$ is called a virtual retract of $G$, if there exists a finite index subgroup $K\leq G$ such that $H\leq K$ and $H$ is a retract of $K$. 
\end{definition}

\begin{definition}
A group $G$ is called cyclically  virtually retractible if for any infinite
cyclic subgroup $\langle a\rangle \leq G$ there is a finite-index subgroup $H\leq G$
and an epimorphism $f:H\rightarrow \mathbb{Z}$ such that $f|_{H\cap \langle
a\rangle }: H\cap \langle
a\rangle\to \mathbb{Z}$ is an isomorphism.
\end{definition}

Clearly, a $\mathrm{(VRC)}$ group is cyclically virtually
retractable. Hagen and Wise \cite[Theorem 3.4]{hw}  prove that a virtually
special group is cyclically virtually retractible.

\begin{lemma}
\label{lemresi}A torsion-free cyclically  virtually retractible group $G$ is residually
finite.
\end{lemma}

\begin{proof}
For any $1\neq g\in G,$ there is a finite index $H\leq G$ and an epimorphism $%
f:H\rightarrow \mathbb{Z}$ such that $f(H\cap \langle g\rangle )=\mathbb{Z}.$
If $g\in H,$ the preimage $f^{-1}(2\mathbb{Z})$ is a finite index subgroup
of $H$ does not contain $g.$ If $g\notin H,$ we are done.
\end{proof}

\begin{lemma}
\label{lemvcr}A torsion-free group $G$ is cyclically  virtually retractible if and only if $%
G$ is $\mathrm{(VRC)}$.
\end{lemma}

\begin{proof}
It is obvious that a $\mathrm{(VRC)}$ group is cyclically virtually retractible. Conversely, for a
torsion-free cyclically retractible group $G,$ it is residually finite by Lemma %
\ref{lemresi}. For any infinite cyclic subgroup $\langle a\rangle,$ there is a finite index
subgroup $H\leq G$ and an epimorphism $f:H\rightarrow \mathbb{Z}$ such that $%
f|_{H\cap \langle a\rangle} :H\cap \langle
a\rangle\to \mathbb{Z}$ is an isomorphism. This implies that $H\cap \langle a\rangle$
is a retract of $H.$ Minasyan \cite[Theorem 1.4]{min}  shows that any
finite index extension $\langle a\rangle\geq H\cap \langle a\rangle$ is a virtually retract
of $G.$ Therefore, $G$ is $\mathrm{(VRC)}$.
\end{proof}

\begin{remark}
    In general, a group has $\mathrm{(VRC)}$ if and only if it is residually finite and cyclically virtually retractible by  \cite[Lemma 2.3 and Lemma 3.4]{min}.
\end{remark}

Recall that an element $g\in G$ is a \emph{virtual retractor} if there is some finite index subgroup $H$ of $G$ such that $\langle g\rangle \subseteq H$, and $H$ retracts to $\langle g\rangle$

\begin{theorem}
\label{th9}
Let $G$ be a residually finite free-by-cyclic group and $a\in G$ is not a virtual retractor. Then the amalgamated product $%
G\ast _{ a =t }(F_{2}\times \mathbb{Z})$
(where $t$ is a generator of $\mathbb{Z})$ is not virtually free-by-cyclic.
\end{theorem}

\begin{proof}
Since $a$ is not a virtual retractor, we have for any finite index subgroup $H\leq G
$ and any epimorphism $\phi :H\rightarrow \mathbb{Z}$ the intersection $%
\langle a\rangle \cap H$ cannot be mapped to be $\mathbb{Z}$ by $\phi $ via isomorphism (cf. \cite[Theorem 1.4]{min}).

Now suppose that a finite-index subgroup $\Gamma \leq G\ast _{a =t }(F_{2}\times \mathbb{Z})$ is free-by-cyclic. Choose an epimorphism $f:\Gamma \rightarrow 
\mathbb{Z}$ with a free kernel.  We have $\Gamma \cap G$ is finite index in $G$. If $f|_{\Gamma
\cap \langle a\rangle }$ is injective, then $\mathrm{Im}f |_{\Gamma \cap
 \langle a\rangle }$ is of finite index in the image $\mathbb{Z}$ and $%
f|_{G\cap \Gamma }^{-1}(\mathrm{Im}f |_{\Gamma \cap \langle a\rangle })\leq G$
would be of finite index, a contradiction to the choice of $a$. If $f_{\Gamma
\cap \langle a\rangle }$ is not injective, we have that $f(a^{k})=0$
for some nonzero integer $k.$ Then $f(t^{k})=0.$ But this means $\ker f$ contains the
subgroup $\ker f|_{F_{2}}\times \langle t^{k}\rangle ,$ which shows that $%
\ker f$ is not free. Therefore, any finite index subgroup $\Gamma $ is not
free-by-cyclic.
\end{proof}

\begin{example}
Let $G=F_{\infty }\rtimes \mathbb{Z}$ be the infinitely generated  non-residually finite group
constructed by Baumslag \cite{baum}. By Lemma \ref{lemresi}, the group $G$ is not
cyclically  virtually retractible. The proof of Theorem \ref{th9} implies that there is an
element $a\in G$ such that the amalgamated product $G\ast _{a =t }(F_{2}\times \mathbb{Z})$, is not virtually
free-by-cyclic. 
\end{example}

The following is an observation about $\mathrm{(VRC)}$ groups.
\begin{lemma}
\label{vrc-euclid}
    Let $G$ be a $\mathrm{(VRC)}$ group. For any infinite-order element $g\in G$, there is an Euclidean space $\mathbb{E}^k$ on which $G$ acts isometrically with $g$ acting as a nontrivial translation on some geodesic line.
\end{lemma}
\begin{proof}
    Suppose that there is a finite index subgroup $\Gamma
\leq G$ and a retraction $r:\Gamma \rightarrow
\langle g\rangle .$ Let $\langle g\rangle \cong \mathbb{Z}$ act on
the real line $\mathbb{R}$ by translations. Let $\{g_{i}\}_{i=1}^{k}$ be a
system of representatives of left cosets $G/\Gamma $ (with $g_{1}=1$) and $%
\Pi _{G/\Gamma }\mathbb{R}$ the real vector space spanned by the basis $%
\{g_{i}\}_{i=1}^{k}.$  The group $G$ acts on the product $\mathbb{E}^k=\Pi _{G/\Gamma }%
\mathbb{R}$ by $s(\sum g_{i}x_{i})=\sum g_{j}h_{i}x_{i}$ assuming $%
sg_{i}=g_{j}h_{i}$ for some $h_{i}\in \Gamma .$ Since the action of each element $s$ is a composite of translations and permuting coordinates, the action is isometric. Now by definition, $g$ acts on $\mathbb{E}^k$ with unbounded orbits, it must act on  $\mathbb{E}^k$ as a hyperbolic isometry since every isometry of $\mathbb{E}^k$ is semisimple \cite[Proposition II.6.5]{BH}.

\end{proof}

\begin{lemma}
\label{lem6.11}
Let $G = G(\{p_{i},q_{i}\}_{i=1}^{n-1})$ be the $F_n$-by-$\mathbb{Z}$ group
defined  in the introduction. Suppose that there are no vectors $%
v, w\in \mathbb{R}^{2}$ such that $w\neq 0$ and
\begin{equation*}
\Vert v-p_{i}w\Vert =\Vert v+q_{i}w\Vert ,i=1,2,...,n-1.
\end{equation*}%
Then $G(\{p_{i},q_{i}\}_{i=1}^{n-1})$ does not virtually retract onto the
cyclic subgroup $\langle a_{0}\rangle $. In particular, when $%
n=3, p_{1}=p_{2}=0,q_{1}\neq \pm q_{2}$ are nonzero integers, $a_{0}$ is not a virtual retractor of $G$.
\end{lemma}

\begin{proof}
Suppose that there is a finite index subgroup $\Gamma
\leq G({\{p_{i},q_{i}\}_{i=1}^{n-1}})$ and a retraction $r:\Gamma \rightarrow
\langle a_{0}\rangle .$ Let $\langle a_{0}\rangle \cong \mathbb{Z}$ act on
the real line $\mathbb{R}$ by translations. By Lemma \ref{vrc-euclid} (and its proof), there is an Euclidean space $\mathbb{E}^{k'}$ on which $G$ acts isometrically with $a_0$ acting as a nontrivial translation  (on some geodesic line). The  flat torus
theorem (for general actions) \cite[Theorem II.7.20]{BH} implies that there is an
Euclidean space $\mathbb{E}^{k}\subseteq \mathrm{Min}(\langle a_{0},t\rangle
)\subseteq \Pi _{G/\Gamma }\mathbb{R}$ ($k\leq 2$), on which $\langle
a_{0},t\rangle \cong \mathbb{Z}^{2}$ acts cocompactly by translations.  Let $%
v,w$ be the translation vectors of $t,a_0$ respectively. Since $a_{0}$ acts by
non-trivial translation on $\mathbb{R}$, we see that $w\neq 0$ (and thus $%
k>0)$. Since $a_{0}^{-p_{i}}t=a_{i}a_{0}^{q_{i}}ta_{i}^{-1},$ we have $$\Vert
v-p_{i}w\Vert =\Vert v+q_{i}w\Vert $$ for each $i=1,2,...,n.$ (When $k=2,$
the vectors are from $\mathbb{R}^{2}.$ When $k=1,$ we view $\mathbb{R}$ as
a subspace of $\mathbb{R}^{2}$). This is a contradiction to the assumption. Since a circle (centered at the origin and of radius $|v|$) in the plane cannot intersect the line $v+\mathbb{R}w$ at three distinct points, $G$ cannot satisfy the condition when $
n=3, p_{1}=p_{2}=0,q_{1}\neq \pm q_{2}$. 
\end{proof}

The following example answers Question \ref{prob3} in the negative.

\begin{example}\label{corlast}
    Let $G=\langle a,b,c,t:tat^{-1}=a,tbt^{-1}=ba, tct^{-1} =ca^2\rangle $ be the Gersten group. Then  $G\ast _{a=t}(F_{2}\times \mathbb{Z})$ (where $t$ is
the generator of $\mathbb{Z})$ is not virtually free-by-cyclic. In
particular, the family of virtually free-by-cyclic groups is not closed
under taking amalgamation along cyclic subgroups. 
\end{example} 

\begin{proof}
Lemma \ref{lem6.11} shows that $\langle a\rangle $ is not a virtual retractor of $G.$
Theorem \ref{th9} implies that $G\ast _{a=t}(F_{2}\times \mathbb{Z})$
is not virtually free-by-cyclic.
\end{proof}

We can also improve \Cref{corlast} to show that Condition (2) in \Cref{thm:vir-free-nec} is not sufficient.

\begin{example}\label{corlast2}
    Let $G_1 =\langle a,b,c,t\mid tat^{-1}=a,tbt^{-1}=ba, tct^{-1} =ca^2\rangle $ be the Gersten group, $G_2 = \langle x,s\mid sxs^{-1} = x \rangle \cong \mathbb{Z}^2$. Then the tubular group $G = G_1 \ast_{a=x} G_2$ is not virtually free-by-cyclic. 
\end{example}
\begin{proof}
    Note first that $G$ has the following tubular presentation:
    $$ \langle a,b,c, t,s \mid [a,t]=1,btb^{-1}=at, ctc^{-1} =a^2t, sas^{-1} =a\rangle .$$
    Let $T$ be the corresponding Bass--Serre tree.  Suppose that $H$ is a finite index subgroup of $G$. Since $H$ also acts cocompactly on $T$ with vertex stabilizers isomorphic to $\mathbb{Z}^2$ and edge stabilizers isomorphic to $\mathbb{Z}$, the group $H$ is tubular. Let $H\cap \langle a\rangle = \langle a^m\rangle$ for some $m\geq 1$. This means some edge group of the tubular group $H$ will be $\langle a^m\rangle$. Suppose that $H$ is a free-by-cyclic group. By \cite[Theorem 2.1]{bu}, there is a surjection $\phi: H \to \mathbb{Z}$ such that the restriction of $\phi$  to every edge group is injective.  In particular $\phi\mid_{H\cap \langle a\rangle} = \phi\mid_{\langle a^m\rangle}$ is injective. But this means  the restriction map $\phi: G_1\cap H \to \mathbb{Z}$ has $\phi(a^m) \neq 0$. Let $\bar{G}_1$ be the preimage of $\phi(a^m)$ in $G_1\cap H$. Then $\bar{G}_1$ is a finite index subgroup of $G_1\cap H$ hence $G_1$, and $a^m$ is a retractor of $\bar{G}_1$. By \cite[Theorem 1.4]{min}, we have that $a$ is also a virtual retractor of $G_1$. But this contradicts \Cref{lem6.11}.

\end{proof}

\section{$\mathrm{(VRC)}$ and cyclic subgroup separability}

The purpose of this section is to provide a counterexample to Question \ref{Ques-Min} of Minasyan \cite{min}. We first need a lemma on the separability of the cyclic subgroups of mapping tori. The proof of the following lemma is inspired by that of Hughes--Kudlinska \cite[Proposition 2.7]{hk}.

\begin{theorem}
\label{lastthm}
Let $K$ be a finitely generated group whose cyclic subgroups are separable
and $\phi :K\rightarrow K$  be an automorphism. Then any cyclic subgroup of $G=K\rtimes _{\phi }\mathbb{Z}$ is separable.
\end{theorem}

\begin{proof}
Let $t$ be a generator of $\mathbb{Z}$ and suppose that the cyclic subgroup is generated by $g=g_{0}t^{m}$. Now choose an arbitrary element $x=x_{0}t^{n}\in G$ such that $x\not\in \langle g\rangle$. It suffices now to prove that $x$ is separable from the cyclic subgroup $\langle g\rangle $.

If $m\neq 0$,  $g$ is a generalized retractor since the natural projection $%
p:K\rtimes _{\phi }\mathbb{Z\rightarrow Z}$ has $p(g)\neq 0.$ Recall that a cyclic subgroup separable group is residually finite, so is $G$ (as a splitting cyclic extension of $K$). If $|m|$=1, $g$ is a retractor and $G$ is residually finite. So by \cite[Lemma 3.9]{hw99}), $\langle g\rangle $ is separable in $G$. If $|m|>1$, by \cite[Lemma 3.9]{hw99}), we have that $\langle g\rangle $ is separable in finite index subgroup $K\rtimes \langle t^m\rangle$ of $G$. This means $\langle g\rangle $ can be written as a family of finite index subgroups $\{H_i\}_{i\in I}$ of $K\rtimes \langle t^m\rangle$. Hence $\langle g\rangle  = K\rtimes \langle t^m\rangle \cap (\cap_{i\in I} H_i)$. Since $K\rtimes \langle t^m\rangle$ is a finite index subgroup of $G$, $H_i$ are also finite index subgroups of $G$. This shows that $\langle g\rangle $ is separable in $G$.

If $m=0,$ and $n\neq 0.$ Take 
\begin{equation*}
f:K\rtimes _{\phi }\mathbb{Z}~\overset{p}{\mathbb{\rightarrow }} ~\mathbb{%
Z\rightarrow Z}/2n
\end{equation*}%
 to be the composition of $p$ with the natural epimorphism onto $\mathbb{Z}/2n$. We
have $f(g)=0$ but $f(x)\neq 0.$ The kernel $\ker f$ is a finite index
subgroup separating $g$ and $x.$

It remains to consider the case when $m=0$ and $n=0.$ Since $x\notin
\langle g\rangle $ and $K$ is cyclic subgroup separable, there is a finite index subgroup $H$ of $K$ such that $%
g\in H$ but $x\notin H.$ Since $K$ is finitely generated, the set $S$ of
subgroups in $K$ of index $[K,H]$ is finite. The automorphism $\phi $ acts
on this finite set $S$. So there is an integer $M$ such that $\phi ^{M}(H)=H.$ Now the
finite index subgroup $\langle H,t^{M}\rangle \leq G$ contains $g,$ but not $x.$
\end{proof}

Since finitely generated free groups are cyclic subgroup separable, we have the following corollary (see also \cite[Corollary 5.3]{BM06}, \cite[Proposition 2.7]{hk}). 
\begin{corollary}\label{cor:css}
$F_n$-by-$\mathbb{Z}$ groups are cyclic subgroup separable. 
\end{corollary}

The following example answers Question \ref{Ques-Min} in the negative, by combining Lemma \ref{lem6.11} and Corollary \ref{cor:css}.

\begin{example}
\label{eg7.2}
    Let $G = \langle a,b,c,t \mid tat^{-1}=a,tbt^{-1}=ba, tct^{-1}=ca^2 \rangle$ be the Gersten group. It  has a tubular presentation:
    $$\langle a,b,c,t \mid [a,t]=1,b^{-1}tb=at, c^{-1}tc=a^2t \rangle.$$
    It is cyclic subgroup separable by \Cref{cor:css} but does not have $\mathrm{(VRC)}$ by \Cref{lem6.11}. 
\end{example}

\section{RFRS vs. virtually RFRS, comments and questions}
In this section, we first discuss the subtle difference between RFRS and virtually RFRS, then we provide some open questions. Note that the original version \cite[Problem 5.3]{page} of Question \ref{prob1.1} is stated in terms of RFRS instead of virtually RFRS. We first show that a $\mathrm{CAT(0)}$ free-by-cyclic may not be RFRS.

Recall that in \cite{Ag08} a finitely generated group $G$ is called RFRS if  it has a sequence of finite index normal subgroups $G=G_{0}\geq G_{1}\geq \cdots$
such that: 1) $\cap G_{i}=1$; 2) $G_{i+1}\geq \ker (G_{i}\rightarrow
H_{1}(G_{i})\bigotimes \mathbb{Q})$ for each $i.$ A virtually special group
is virtually RFRS \cite[Theorem 2.2]{Ag08}.

Recall that RFRS groups are locally indicable; see for example \cite[Proposition 2.7]{fish}. The following lemma strengthens this result. Actually, taking $H$ to be finitely generated and taking no additional elements $g_i$, that is, $k=0$, recovers local indicability.

\begin{lemma}
\label{8.1}
Let $G$ be a $\mathrm{RFRS}$ group. For any (possibly  infinitely generated) non-trivial subgroup  $H$ and
any finitely many elements $g_{1},...,g_{k}\in G,$ there are positive
integers $n_{1},n_{2},...,n_{k}$ such that there is a surjection $f:\langle
H,g_{1}^{n_{1}},...,g_{k}^{n_{k}}\rangle \rightarrow \mathbb{Z}$ satisfying $%
f(H)\neq 0.$
\end{lemma}

\begin{proof}
 Let $n$ be the integer such that $H\leq G_{n}$ but $H$ is not a subgroup of $%
G_{n+1}.$ Since $G_{n}\leq G$ is finite index, there are integers $n_{i}$
satisfying $g_{i}^{n_{i}}\in G_{n}.$ The map $f:G_{n}\rightarrow
H_{1}(G_{n};\mathbb{Z})\rightarrow H_{1}(G_{n};\mathbb{Q})$ has $f(H)$
nontrivial, since $\ker f\leq G_{n+1}.$ Now a subgroup of the finitely generated
 abelian group $$H_{1}(G_{n};\mathbb{Z})=H_{1}(G_{n};\mathbb{Z}%
)_{tor}\oplus H_{1}(G_{n};\mathbb{Z})_{free}$$ is finitely generated.
The image of $H$ in the free-part $H_{1}(G_{n};\mathbb{Z})_{free}$ is
nontrivial and thus has a surjection onto $\mathbb{Z}.$ The restriction $f|_{\langle
H,g_{1}^{n_{1}},...,g_{k}^{n_{k}}\rangle}$ satisfies the required property.
\end{proof}

\begin{example}
($\mathrm{CAT(0)}$ free-by-cyclic may not be $\mathrm{RFRS}$) Let $G=F_{2}\rtimes \mathbb{Z}=\langle
a,b,t:tat^{-1}=ab,tbt^{-1}=ab^{2}\rangle .$ Choose $H=\langle F_{2}\rangle.$ For any
positive integer $n,$ the subgroup $\langle H,t^{n}\rangle $ does not have a
surjection onto $\mathbb{Z}$ with nontrivial image of $F_2.$ Therefore, the group $G$ is not $\mathrm{RFRS}$ (but is indeed virtually $\mathrm{RFRS}$).
\end{example}

In the same vein, Agol asked in his ICM survey paper \cite[Question 11]{Ag14} whether the braid groups $B_n$ are RFRS. Note that any non-trivial subgroup of RFRS groups has a surjection onto $\mathbb{Z}$ by Lemma \ref{8.1}. On the other hand, the commutator subgroups of the braid groups with at least five strings are finitely presented and perfect \cite[Theorem 2.1]{GL69}, so $B_n$ cannot be RFRS if $n\geq 5$ by Lemma \ref{8.1} or \cite[Proposition 2.7]{fish}. It is still open though whether braid groups are virtually RFRS.

\begin{question}
\label{8.3}
    Let $G$ be a $\mathrm{CAT(0)}$ free-by-cyclic group. Is every finitely generated subgroup of $G$ again $\mathrm{CAT(0)}$?
\end{question}
Since a subgroup of a virtually special (resp. virtually RFRS) group is virtually special (resp. virtually RFRS), a negative answer to Question \ref{8.3} would answer the other direction of Question \ref{prob1.1}. 

At this point, it is hard to attack problems related to (infinitely generated free)-by-cyclic groups. It is natural to ask whether such problems can sometimes be reduced to the finite-rank case (we learned the first part of Question \ref{ques:embed} from Richard Wade in May 2023, see also \cite[Conjecture 1.2]{CW24}). 

\begin{question}\label{ques:embed}
    Does a finitely generated  free-by-cyclic group embed into some $F_n$-by-$\mathbb{Z}$ group? Does a hyperbolic finitely generated   free-by-cyclic group embed into some hyperbolic $F_n$-by-$\mathbb{Z}$ group?
\end{question}

\begin{remark}
    Marco Linton has recently  answered this question in the positive \cite{Ly25}.
\end{remark}

It is well-known that we have the following chain of inclusions:
$\mathcal{F}_{vsp}:=$\{virtually special group\} $\subseteq $ $\mathcal{F}%
_{\mathrm{vRFRS}}:=$\{virtually RFRS group\} $\subseteq \mathcal{F}_{\mathrm{VRC}}:=$ \{VRC\} $%
\subseteq \mathcal{F}_{css}:=$ \{Cyclic subgroup separable group\}. Gersten's example
shows that a free-by-cyclic group may not be $\mathrm{(VRC)}$ (see Lemma \ref{lem6.11}), in particular it can not be virtually $\mathrm{RFRS}$. Minasyan \cite[Question 11.5]{min}  asked whether the property $\mathrm{(VRC)}$ is stable under taking amalgamated free products over (virtually) cyclic subgroups.  It is natural to ask the
following question.

\begin{question}
Are the inclusions $\mathcal{F}_{vsp}\subseteq $ $\mathcal{F}_{\mathrm{vRFRS}}\subseteq 
\mathcal{F}_{\mathrm{VRC}}$ proper for the class of finitely generated free-by-cyclic groups? Is
the family $\mathcal{F}_{vsp}$ (resp.  $\mathcal{F}_{\mathrm{vRFRS}},\mathcal{F}_{\mathrm{VRC}}$
) closed under taking amalgamated product over cyclic subgroup
among finitely generated free-by-cyclic groups?
\end{question}


\end{document}